\documentclass[reqno]{amsart}
\usepackage{amssymb}
\usepackage[usenames, dvipsnames]{color}
\usepackage{hyperref}

\usepackage{verbatim}
\usepackage{todonotes}
\usepackage[pagewise]{lineno}
\numberwithin{equation}{section}

\newtheorem{theorem}{Theorem}[section]
\newtheorem{corollary}[theorem]{Corollary}
\newtheorem{lemma}[theorem]{Lemma}
\newtheorem{proposition}[theorem]{Proposition}

\theoremstyle{definition}
\newtheorem{remark}[theorem]{Remark}

\theoremstyle{definition}

\theoremstyle{definition}
\newtheorem{assumption}[theorem]{Assumption}

\makeatletter
\def\dashint{\operatorname%
{\,\,\text{\bf--}\kern-.98em\DOTSI\intop\ilimits@\!\!}}
\makeatother

\def\bH{\mathbb{H}}

\def\bR{\mathbb{R}}

\def\cA{\mathcal{A}}
\def\cB{\mathcal{B}}
\def\cC{\mathcal{C}}
\def\cD{\mathcal{D}}

\def\cF{\mathcal{F}}
\def\cG{\mathcal{G}}
\def\cH{\mathcal{H}}

\def\cM{\mathcal{M}}

\def\cS{\mathcal{S}}

\makeatletter
\@namedef{subjclassname@2020}{%
  \textup{2020} Mathematics Subject Classification}
\makeatother

\begin{document}
\title[Fractional parabolic equations]{Time fractional parabolic equations with measurable coefficients and embeddings for fractional parabolic Sobolev spaces}

\author[H. Dong]{Hongjie Dong}
\address[H. Dong]{Division of Applied Mathematics, Brown University, 182 George Street, Providence, RI 02912, USA}

\email{Hongjie\_Dong@brown.edu}

\thanks{H. Dong was partially supported by the Simons Foundation, grant \# 709545.}

\author[D. Kim]{Doyoon Kim}
\address[D. Kim]{Department of Mathematics, Korea University, 145 Anam-ro, Seongbuk-gu, Seoul, 02841, Republic of Korea}

\email{doyoon\_kim@korea.ac.kr}

\thanks{D. Kim was supported by the National Research Foundation of Korea (NRF) grant funded by the Korea government (MSIT) (2019R1A2C1084683).}



\begin{abstract}
We consider time fractional parabolic equations in both divergence and non-divergence form when the leading coefficients $a^{ij}$ are measurable functions of $(t,x_1)$ except for $a^{11}$ which is a measurable function of either $t$ or $x_1$. We obtain the solvability in Sobolev spaces of the equations in the whole space, on a half space, or on a partially bounded domain. The proofs use a level set argument, a scaling argument, and embeddings in fractional parabolic Sobolev spaces for which we give a direct and elementary proof.
\end{abstract}

\maketitle

\section{Introduction}

We study time fractional parabolic equations of the form
\begin{align}
							\label{eq0218_01}
-\partial_t^\alpha u + D_i(a^{ij}D_ju) &= D_i g_i + f,\\
							\label{eq0218_02}
-\partial_t^\alpha u + a^{ij}D_{ij}u &= f,
\end{align}
where $\partial_t^\alpha u$ is the Caputo fractional time derivative of order $\alpha \in (0,1)$:
$$
\partial_t^\alpha u(t,x) = \frac{1}{\Gamma(1-\alpha)} \frac{d}{dt} \int_0^t (t-s)^{-\alpha} \left[ u(s,x) - u(0,x) \right] \, ds.
$$
See Section \ref{sec2} for a precise definition of $\partial_t^\alpha u$.
We prove that, for given $f, g_i \in L_p((0,T) \times \bR^d)$, $1<p<\infty$, $d \geq 1$, there exist unique solutions in parabolic Sobolev spaces $\cH_{p,0}^{\alpha,1}$ and $\bH_{p,0}^{\alpha,2}$ (see Section \ref{sec2} for the definitions of $\cH_{p,0}^{\alpha,1}$ and $\bH_{p,0}^{\alpha,2}$) to the equations \eqref{eq0218_01} and \eqref{eq0218_02} with the zero initial condition.
The main contribution of this paper is that the coefficients $a^{ij} = a^{ij}(t,x_1)$, as functions of $(t,x_1) \in \bR \times \bR$, require no regularity assumptions except the uniform ellipticity condition.

In \cite{MR3899965} we proved the unique solvability of the non-divergence type equation \eqref{eq0218_02} in parabolic Sobolev spaces $\bH_{p,0}^{\alpha,2}$ when $a^{ij}(t,x)$ are merely measurable in $t$ (i.e., no regularity assumptions as functions of $t$) and have locally small mean oscillations in $x \in \bR^d$.
For the divergence type equation \eqref{eq0218_01}, we obtained in \cite{MR4030286} the corresponding solvability results when the coefficients $a^{ij}(t,x_1,x')$ are merely measurable in $x_1 \in \bR$ and have locally small mean oscillations in $(t,x')$, $x' \in \bR^{d-1}$.
In this paper, we deal with both divergence and non-divergence type equations when the coefficients are merely measurable both in $t$ and one spatial variable, say, $x_1 \in \bR$, except $a^{11}$, which needs to be a function of either $t$ or $x_1$, that is, $a^{11}=a^{11}(t)$ or $a^{11} = a^{11}(x_1)$.
Considering that we require no regularity assumptions on the coefficients, the class of leading coefficients in this paper is strictly larger than those in \cite{MR3899965, MR4030286} as long as the coefficients $a^{ij}$ are independent of $x' \in \bR^{d-1}$.
In fact, one can also consider coefficients $a^{ij}(t,x_1,x')$ that are measurable in $(t,x_1)$, except $a^{11}$, and have an appropriate regularity assumption as functions of $x'$, such as the small mean oscillation condition, so that the coefficients in \cite{MR3899965, MR4030286} can be covered as proper subclasses.
Such coefficients for time fractional parabolic equations will be discussed in a forthcoming paper in a much general weighted parabolic Sobolev space setting.
In this paper, to present how to deal with time fractional parabolic equations with merely measurable coefficients in the simplest setting, we consider the case that $a^{ij}$ are functions of only $(t,x_1)$.
As an application of the main results, we also prove the solvability of the equations \eqref{eq0218_01} and \eqref{eq0218_02} on a half space $(0,T) \times \bR^d_+$, where $\bR^d_+ = \{(x_1,x') \in \bR^d: x_1 > 0\}$, and on a spatially partially bounded domain $(0,T) \times (0,R) \times \bR^{d-1} = \{(t,x_1,x'): t \in (0,T), x_1 \in (0,R), x' \in \bR^{d-1}\}$.
In particular, the unique solvability results for the equations on the spatially partially bounded domain will be used to build the solvability theory for equations with variable coefficients in $(t,x) \in \bR \times \bR^d$ under appropriate regularity assumptions as functions of the remaining variables $x' \in \bR^{d-1}$, especially when the equations are studied in the frame work of weighted parabolic Sobolev spaces.

Among the advantages of having merely measurable coefficients, particularly noteworthy is that the solvability of equations with coefficients measurable in one spatial variable in the whole Euclidean space gives without much effort (by only using a simple extension argument) the corresponding result for equations on a half space, and further for those on a bounded domain if the boundary is sufficiently smooth.
See Theorem \ref{thm0112_1} and Proposition \ref{prop0118_1} in this paper.
Because of mathematical interests and various potential applications including the boundary value problems mentioned above, classes of merely measurable coefficients have been actively considered for elliptic equations and the usual parabolic equations, that is, the equations as in \eqref{eq0218_01} and \eqref{eq0218_02} with $\partial_t^\alpha u$ replaced with the usual time derivative term $u_t$.
For parabolic equations with $u_t$, in \cite{MR2304157, MR2352490} Krylov considered both divergence and non-divergence type equations with the coefficients $a^{ij}(t,x)$ having no regularity assumptions in $t$ and locally small mean oscillations in $x \in \bR^d$.
Coefficients $a^{ij}(t,x_1,x')$ merely measurable in $x_1 \in \bR$ are considered in \cite{MR2300337, MR2896169} for non-divergence type parabolic equations.
As to divergence type parabolic equations, see, for instance, \cite{MR2764911, MR2968240}.
For more information about elliptic and parabolic equations with measurable coefficients, see the review paper \cite{MR4156495} and reference therein.
Considering the results available in the literature for elliptic and parabolic equations with merely measurable coefficients, it is important to examine to what extends one can remove regularity assumptions on the coefficients for time fractional parabolic equations as in \eqref{eq0218_01} and \eqref{eq0218_02}.
Thus, what we present in this paper is an affirmative answer that the  coefficients $a^{ij}$, as functions of $(t,x_1) \in \bR^2$, for time fractional parabolic equations can be as general as those for the usual parabolic equations.

Regarding our restriction that $a^{11}$ needs to be either $a^{11}(t)$ or $a^{11}(x_1)$, Krylov showed in \cite{MR3488249} that it is not possible to obtain a unique solvability result for $p \in (1,3/2)$ or $p \in (3, \infty)$ if all of the coefficients $a^{ij}$, $i,j=1,\ldots,d$, are functions of $(t,x_1) \in \bR \times \bR$ with no regularity assumptions even when $d=1$.
As such, the class of coefficients in this paper is optimal.

Parabolic equations with time fraction derivatives have many important applications in probability and mechanics.
For instance, such equations can be used to model anomalous diffusions.
See \cite{MR3581300} and references therein.
In particular, there has been study of time fractional parabolic equations in Sobolev spaces in the papers  \cite{MR2125407, MR2538276, MR3581300, MR4097256, MR3899965, MR4030286, MR4186022}.
In \cite{MR2125407}, equations as in \eqref{eq0218_02} and Volterra type equations are considered in mixed-norm Lebesgue spaces with time independent sectorial operators.
Equations as in \eqref{eq0218_01} with more general time derivatives are studied in \cite{MR2538276} in the Hilbert space setting.
See also \cite{MR1164666,MR1143209} for earlier work.
The authors of \cite{MR3581300} studied equations as in \eqref{eq0218_02} in mixed $L_p$ spaces when $a^{ij}(t,x)$ are uniformly continuous in $x \in \bR^d$ and piecewise continuous in $t \in \bR$.
These results are extended in \cite{MR4097256} for the time fractional heat equation in {\em weighted} mixed $L_p$ spaces.
Note that in \cite{MR3581300, MR4097256} the order $\alpha$ of the Caupto fractional derivative belongs not only to $(0,1)$, but also to the hyperbolic regime $(1,2)$.
As mentioned earlier, when $\alpha \in (0,1)$, non-divergence and divergence type equations with coefficients measurable in $t$ or $x_1$ are studied in \cite{MR3899965, MR4030286}, respectively, in $L_p$ spaces and, in \cite{MR4186022} the solvability results are obtained for non-divergence type equations in mixed Sobolev spaces with Muckenhoupt weights.
Thus, this paper can be considered as a sequel to the papers \cite{MR3899965, MR4030286, MR4186022} pursuing the unique solvability of divergence and non-divergence type equations in Sobolev spaces with or without weights when the order $\alpha$ of the Caupto fractional derivative is in the parabolic regime $(0,1)$.

To prove the main results of this paper, we use the level set method utilized in \cite{MR3899965, MR4030286}.
However, since the coefficients  $a^{ij}(t,x_1)$ are merely measurable, we only obtain estimates for $D_{x'}u$, where $D_{x'}= D_{x_j}$ for $j=2,\ldots,d$ when dealing  with the divergence type equation \eqref{eq0218_01}.
To obtain estimates for the full gradient, we use a scaling argument (see Lemma \ref{lem0812_2}) so that the estimates for the full gradient $D_{x}u$ can be obtained from those for $D_{x'}u$ along with the right-hand side of the equation.
Once we resolve the divergence type equation, we derive desired estimates  for non-divergence type equations from those for divergence type equations upon the observation that the spatial derivatives of solutions to equations in non-divergence form can be regarded as solutions to equations in divergence form.
In particular, following the argument in \cite{MR2833589} we turn equations in non-divergence form into those in divergence form with non-symmetric $a^{ij}$.
We here remark that the coefficients $a^{ij}$ in this paper are not necessarily symmetric.
In the case $a^{11}=a^{11}(x_1)$, we also use a change of $x_1$ variable. Combining all these different ingredients in a proof is quite delicate.
To obtain the unique solvability results of equations on a half space and on a spatially partially bounded domain, we make use of (odd, even, and periodic) extension arguments, for which it is essential that the coefficients $a^{ij}$ have no regularity assumptions so that they can be discontinuous when they are periodically extended.

In this paper, we also present detailed proofs of embeddings for parabolic Sobolev spaces $\cH_{p,0}^{\alpha,1}$, which serve as solutions spaces to time fractional parabolic equations in divergence form.
See Theorem \ref{thm0717_1} for H\"{o}lder embeddings and Theorem \ref{thm0811_01} for Sobolev embeddings.
It is clear that such embeddings are necessary when studying PDEs in the framework of Sobolev spaces.
We use the embeddings in the proof of Proposition \ref{prop0118_1} in this paper.
Embeddings for $\bH_{p,0}^{\alpha,2}$ and $\bH_{p,0}^{\alpha,1}$ are proved in \cite{MR3899965, MR4030286} with somewhat less optimal exponents.
On the other hand, one can also find embeddings for fractional Sobolev spaces in \cite{MR3000457, MR3965393}, which are based on the so called {\em mixed derivative theorem} (see \cite{MR0482314} and \cite[Proposition 3.2]{MR2318575}) with vector-valued Triebel-Lizorkin type spaces and their interpolation spaces.
We believe that elementary and self-contained proofs of the embeddings are of interest to a wider audience, and the last section of this paper is devoted to provide such arguments.
In particular, our proof for Sobolev embedding uses mollifications of functions, which turn out to be very handy when dealing with embeddings or boundary traces  of functions.

The remainder of the paper is organized as follows.
In the next section, we introduce some notation and state the main results of the paper.
In Section \ref{sec3}, we deal with equations with coefficients $a^{ij}$ either being functions of only $t$ or only $x_1$. We also show that the estimates for the full gradient can be obtained from those of $D_{x'}u$ and the right-hand side as long as $a^{11}=a^{11}(t)$ or $a^{11}=a^{11}(x_1)$.
In Section \ref{sec4}, we use the level set method to obtain estimates of $D_{x'}u$ for equations in divergence form.
We complete the proofs of our main results in Section \ref{sec5}.
Equations on a half space and on a spatially partially bounded domain are discussed in Section \ref{sec6}.
In the last section, as noted above, we present self-contained proofs of embeddings for $\cH_{p,0}^{\alpha,1}$.

\section{Notation and Main results}
\label{sec2}

For $\Omega \subset \bR^d$ and $T > 0$, we denote $\Omega_T = (0,T) \times \Omega$.
In particular, we have $\bR^d_T = (0,T) \times \bR^d$.
We write $Du$ or $D_x u$ to denote $D_{x_j}u$, $j=1,2,\ldots,d$, whereas by $D_{x'}u$, we mean $D_{x_j}u$, $j=2,\ldots,d$.
For $\alpha \in (0,1]$, we denote the parabolic cylinders by
\begin{equation*}
				%			\label{eq0210_01}
Q_{r_1, r_2}(t,x) = (t-r_1^{2/\alpha}, t) \times B_{r_2}(x), \quad Q_r(t,x) = Q_{r,r}(t,x).
\end{equation*}
We often write $B_r$ and $Q_r$ for $B_r(0)$ and $Q_r(0,0)$. 
We use the notation $(u)_{\cD}$ to denote the average of $u$ over $\cD$, where $\cD$ is a subset of $\bR^{d+1}$.
For $(t_0,x_0) \in \bR \times \bR^d$ and a function $f$ defined on $(-\infty,T) \times \bR^d$ with $T \in (-\infty,\infty]$, we set
$$
\cM f(t_0,x_0) = \sup_{Q_r(t,x) \ni (t_0,x_0)} \dashint_{Q_r(t,x)} |f(s,y)|1_{(-\infty,T) \times \bR^d} \, dy \, ds
$$
and
$$
\cS\cM f(t_0,x_0) = \sup_{Q_{r_1,r_2}(t,x) \ni (t_0,x_0)}
\dashint_{Q_{r_1,r_2}(t,x)} |f(s,y)|1_{(-\infty,T) \times \bR^d} \, dy \, ds.
$$

For $\alpha \in (0,1)$ and $S \in \bR$, we denote the $\alpha$-th integral
$$
I^\alpha_S \varphi(t) = \frac{1}{\Gamma(\alpha)} \int_S^t (t-s)^{\alpha - 1} \varphi(s) \, ds
$$
for $\varphi \in L_1(S,\infty)$,
where
$$
\Gamma(\alpha) = \int_0^\infty t^{\alpha - 1} e^{-t} \, dt.
$$
In this paper we often write $I^\alpha$ instead of $I_0^\alpha$ for the $\alpha$-th integral with the origin $0$. For a sufficiently smooth function $\varphi(t)$, we set
$$
D_t^\alpha \varphi(t) = \frac{d}{dt} I^{1-\alpha}_S \varphi(t) = \frac{1}{\Gamma(1-\alpha)} \frac{d}{dt} \int_{S}^t (t-s)^{-\alpha} \varphi(s) \, ds
$$
and
\begin{align*}
\partial_t^\alpha \varphi(t) &= \frac{1}{\Gamma(1-\alpha)} \int_{S}^t (t-s)^{-\alpha} \varphi'(s) \, ds\\
&= \frac{1}{\Gamma(1-\alpha)} \frac{d}{dt} \int_{S}^t (t-s)^{-\alpha} \left[ \varphi(s) - \varphi(S) \right] \, ds.
\end{align*}
Note that if $\varphi(S) = 0$, then
$$
D_t^\alpha \varphi =  \partial_t(I_S^{1-\alpha} \varphi) = \partial_t^\alpha \varphi.
$$
Since there is no information about the origin $S$ in the notation $D_t^\alpha$ and $\partial_t^\alpha$, we sometimes write $\partial_t I_S^{1-\alpha}$ in place of $D_t^\alpha$ (or $\partial_t^\alpha$ whenever appropriate) to indicate the origin.

For $1 \le p \le \infty$, $\alpha \in (0,1)$, and $k \in \{1,2,\ldots\}$, we set
$$
\widetilde{\bH}_p^{\alpha,k}\left((S,T) \times \Omega\right) = \left\{ u \in L_p: \partial_t I_S^{1-\alpha} u, \, D^\beta_x u \in L_p, \, 0 \leq |\beta| \leq k
\right\}
$$
with the norm
\begin{equation*}
\|u\|_{\widetilde{\bH}_p^{\alpha,k}\left((S,T) \times \Omega\right)} = \|\partial_t I_S^{1-\alpha}u\|_{L_p\left((S,T) \times \Omega\right)} + \sum_{0 \leq |\beta| \leq k}\|D_x^\beta u\|_{L_p\left((S,T) \times \Omega\right)},
\end{equation*}
where by $\partial_t I_S^{1-\alpha} u$ we mean that there exists $g \in L_p\left((S,T) \times \Omega\right)$ such that
\begin{equation}
							\label{eq0122_01}
\int_S^T\int_\Omega g(t,x) \varphi(t,x) \, dx \, dt = - \int_S^T\int_\Omega I_S^{1-\alpha}u(t,x) \partial_t \varphi(t,x) \, dx \, dt
\end{equation}
for all $\varphi \in C_0^\infty\left((S,T) \times \Omega\right)$.
Recall that $\partial_t I_S^{1-\alpha} u$ can be denoted by $D_t^\alpha u$.
Next, $\bH_p^{\alpha,k}\left((S,T) \times \Omega\right)$ is defined by
\begin{align*}
&\bH_p^{\alpha,k}\left((S,T) \times \Omega\right)\\
&= \left\{ u \in \widetilde{\bH}_p^{\alpha,k}\left((S,T) \times \Omega\right): \text{\eqref{eq0122_01} is satisfied for all}\,\, \varphi \in C_0^\infty\left([S,T) \times \Omega\right)\right\}
\end{align*}
with the same norm as for $\widetilde{\bH}_p^{\alpha,k}\left((S,T) \times \Omega\right)$.
Note that test functions for the space $\bH_p^{\alpha,k}\left((S,T) \times \Omega\right)$ are not necessarily zero at $t = S$. We then define
$\bH_{p,0}^{\alpha,k}((S,T)\times\Omega)$
to be the set of functions in $\bH_p^{\alpha,k}((S,T)\times\Omega)$
each of which is
approximated by a sequence $\{u_n(t,x)\} \subset C^\infty\left([S,T]\times \Omega\right)$ such that $u_n$ vanishes for large $|x|$ and $u_n(S,x) = 0$.

By $w \in \bH_p^{-1}\left((S,T) \times \Omega \right)$ we mean that there exist $f, g_i \in L_p\left((S,T) \times \Omega\right)$, $i =1, \ldots, d$, such that
$$
w = D_i g_i + f
$$
in $(S,T) \times \Omega$ in the distribution sense
and
\begin{align*}
&\|w\|_{\bH_p^{-1}\left((S,T) \times \Omega \right)}\\
&= \inf \left\{ \sum_{i=1}^d\|g_i\|_{L_p\left((S,T) \times \Omega\right)} + \|f\|_{L_p\left((S,T) \times \Omega\right)}: w = D_i g_i + f\right\} < \infty.
\end{align*}
We also write
$$
w = \operatorname{div}g + f,
$$
where $g = (g_1,\ldots,g_d)$.

For $u \in L_p\left((S,T) \times \Omega \right)$, we say $D_t^\alpha u \in \bH_p^{-1}\left((S,T) \times \Omega \right)$
if there exist $f, g_i \in L_p\left((S,T) \times \Omega \right)$, $i=1,\ldots, d$, such that, for any $\varphi \in C_0^\infty\left((S,T) \times \Omega\right)$,
\begin{equation}
							\label{eq0720_01}
\int_S^T \int_\Omega I^{1-\alpha}_S u \, \partial_t \varphi \, dx \, dt = \int_S^T \int_\Omega g_i D_i \varphi \, dx \, dt - \int_S^T \int_\Omega f \varphi \, dx \, dt.
\end{equation}

Let $\tilde{\cH}_p^{\alpha, 1}\left((S,T) \times \Omega \right)$ be the collection of functions $u \in L_p\left((S,T) \times \Omega \right)$ such that $D_x u \in L_p\left((S,T) \times \Omega \right)$ and $D_t^\alpha u \in \bH_p^{-1}\left((S,T) \times \Omega \right)$.
For $u \in \tilde{\cH}_p^{\alpha, 1}\left((S,T) \times \Omega \right)$, if \eqref{eq0720_01} holds for any $\varphi \in C_0^\infty\left([S,T) \times \Omega\right)$, we say that $u \in \cH_p^{\alpha,1}\left((S,T) \times \Omega\right)$ with the norm
$$
\|u\|_{\cH_p^{\alpha,1}\left((S,T)\times\Omega\right)} = \|u\|_{L_p\left((S,T) \times \Omega\right)} + \|D_x u\|_{L_p\left((S,T) \times \Omega\right)} + \|D_t^\alpha u\|_{\bH_p^{-1}\left((S,T) \times \Omega\right)}.
$$
We then define $\cH_{p,0}^{\alpha,1}\left((S,T) \times \Omega\right)$ to be the collection of $u \in \cH_p^{\alpha,1}\left((S,T) \times \Omega\right)$ satisfying the following. There exists a sequence of $\{u_n\} \subset C^\infty\left([S,T] \times \Omega\right)$ such that $u_n$ vanishes for large $|x|$,  $u_n(S,x) = 0$, and
\begin{equation*}
				%			\label{eq0723_01}
\|u_n - u\|_{\cH_p^{\alpha,1}\left((S,T) \times \Omega\right)} \to 0
\end{equation*}
as $n \to \infty$.

Throughout the paper, we assume that the leading coefficients $a^{ij}$ satisfy the uniform ellipticity condition: there exists $\delta \in (0,1)$ such that
$$
a^{ij}\xi_i \xi_j \geq \delta |\xi|^2, \quad |a^{ij}| \leq \delta^{-1}
$$
for any $\xi \in \bR^d$ and $(t,x) \in \bR \times \bR^d$.
The following assumption on $a^{ij}$ means that $a^{ij}$ are independent of $x' \in \bR^{d-1}$ with no regularity assumptions except that $a^{11}$ is either a function of only $t$ or $x_1$.

\begin{assumption}
							\label{assum0808_2}
The coefficients $a^{ij}$ satisfy either (i) or (ii) of the following.
\begin{enumerate}
\item[(i)]
$a^{11} = a^{11}(t)$, $a^{ij} = a^{ij}(t,x_1)$ for $(i,j) \neq (1,1)$.
\item[(ii)]
$a^{11} = a^{11}(x_1)$, $a^{ij} = a^{ij}(t,x_1)$ for $(i,j) \neq (1,1)$.
\end{enumerate}
\end{assumption}

The following is our main result for equations in divergence form.

\begin{theorem}[Divergence case with measurable coefficients]
							\label{thm0808_01}
Let $\lambda \geq 0$, $\alpha \in (0,1)$, $p \in (1,\infty)$, $T \in (0,\infty)$, and $a^{ij}$ satisfy Assumption \ref{assum0808_2}.
Then, for $u \in \cH_{p,0}^{\alpha,1}(\bR^d_T)$ satisfying
\begin{equation}
							\label{eq0814_06}
-\partial_t^\alpha u + D_i\left(a^{ij} D_j u \right)-\lambda u = D_i g_i + f
\end{equation}
in $\bR^d_T$, where $g_i, f \in L_p(\bR^d_T)$, we have, for $\lambda > 0$,
\begin{equation}
							\label{eq0114_02}
\|Du\|_{L_p(\bR^d_T)} + \sqrt{\lambda}\|u\|_{L_p(\bR^d_T)} \leq N_0 \|g\|_{L_p(\bR^d_T)} + \frac{N_0}{\sqrt{\lambda}}\|f\|_{L_p(\bR^d_T)},
\end{equation}
and, for $\lambda \geq 0$,
\begin{equation}
							\label{eq0811_02}
\begin{aligned}
&\|Du\|_{L_p(\bR^d_T)} \leq N_0 \|g\|_{L_p(\bR^d_T)} + N_0T^{\alpha/2}\|f\|_{L_p(\bR^d_T)},
\\
&\|u\|_{L_p(\bR^d_T)} \leq N_0 T^{\alpha/2}\|g\|_{L_p(\bR^d_T)} + N_0 T^{\alpha}\|f\|_{L_p(\bR^d_T)},
\end{aligned}
\end{equation}
where $N_0 = N_0(d,\delta,\alpha, p)$.
Moreover, for any $g_i, f \in L_p(\bR^d_T)$, there exists a unique $u \in \cH_{p,0}^{\alpha,1}(\bR^d_T)$ satisfying \eqref{eq0814_06}.
\end{theorem}

For a domain $\Omega$ in $\bR^d$ and $T \in (S, \infty)$, we say that $u \in \cH_{p,0}^{\alpha,1}\left((S,T) \times \Omega \right)$ satisfies the divergence form equation
$$
- \partial_t^\alpha u + D_i \left( a^{ij} D_j u\right) -\lambda  u = D_i g_i + f
$$
in $(S,T) \times \Omega$,
where $g_i, f \in L_p\left((S,T) \times \Omega \right)$ if
\begin{multline*}
				%			\label{eq0103_01}
\int_S^T \int_{\Omega} I_S^{1-\alpha}u \, \varphi_t \, dx \, dt + \int_S^T \int_{\Omega} \left(- a^{ij}D_j u D_i \varphi -\lambda u \varphi \right) \, dx \, dt
\\
= \int_S^T \int_{\Omega} \left(f \varphi - g_i D_i \varphi \right) \, dx \, dt
\end{multline*}
for all $\varphi \in C_0^\infty\left([S,T) \times \Omega\right)$.

\begin{remark}
							\label{rem0811_1}
Theorem \ref{thm0808_01} holds  when $p = 2$. See Proposition 4.2 in \cite{MR4030286} with the comment above it saying that no regularity assumptions on $a^{ij}$ are needed when $p = 2$.
See also Theorem 2.1 in \cite{MR4030286}, in particular, the $\lambda = 0$ case, where
the estimate \eqref{eq0811_02} is stated so that $N_0$ and $N_0T^{\alpha}$ are replaced with a constant $N$ which also depends on $T$.
Nevertheless, one can turn the constant $N$ into those as in the estimate \eqref{eq0811_02} by using a scaling $T^{-2}u(T^{2/\alpha}t, Tx)$ since the coefficients $a^{ij}$ only need to satisfy the ellipticity condition.
\end{remark}

\begin{remark}
							\label{rem0115_1}
Under the assumptions of Theorem \ref{thm0808_01}, for $u \in \cH_{p,0}^{\alpha,1}(\bR^d_T)$ and $\lambda \geq 0$ satisfying \eqref{eq0814_06},
we also have
\begin{equation*}
%							\label{eq0115_02}
\|u\|_{\cH_p^{\alpha,1}(\bR^d_T)} \leq N \|g\|_{L_p(\bR^d_T)} + N \|f\|_{L_p(\bR^d_T)}
\end{equation*}
provided that $N$ depends not only on $d$, $\delta$, $\alpha$, $p$, $T$, but also on $\lambda$.
The constant $N$ can be chosen as an  increasing function with respect to $\lambda$.
To see this, we estimate $\|u_t\|_{\bH_p^{-1}(\bR^d_T)}$ by using the equation with the estimates in Theorem \ref{thm0808_01}.
\end{remark}

The theorem below is our main result for equations in non-divergence form.

\begin{theorem}[Non-divergence case with measurable coefficients]
							\label{thm0808_02}
Let $\lambda \geq 0$, $\alpha \in (0,1)$, $p \in (1,\infty)$, $T \in (0,\infty)$, and $a^{ij}$ satisfy Assumption \ref{assum0808_2}.
Then, for $u \in \bH_{p,0}^{\alpha,2}(\bR^d_T)$ satisfying
\begin{equation}
							\label{eq0814_04}
-\partial_t^\alpha u + a^{ij} D_{ij} u - \lambda u = f
\end{equation}
in $\bR^d_T$, where $f \in L_p(\bR^d_T)$, we have
\begin{equation}
							\label{eq0814_01}
\|\partial_t^\alpha u\|_{L_p(\bR^d_T)} + \|D^2u\|_{L_p(\bR^d_T)} +\sqrt{\lambda}\|Du\|_{L_p(\bR^d_T)} + \lambda\|u\|_{L_p(\bR^d_T)} \leq N_0\|f\|_{L_p(\bR^d_T)},
\end{equation}
where $N_0 = N_0(d,\delta,\alpha, p)$.
We also have
\begin{equation}
							\label{eq0115_01}
\|u\|_{\bH_p^{\alpha,2}(\bR^d_T)} \leq N\|f\|_{L_p(\bR^d_T)},
\end{equation}
where $N=N(d,\delta,\alpha,p,T)$.
Moreover, for any $f \in L_p(\bR^d_T)$, there exists a unique $u \in \bH_{p,0}^{\alpha,2}(\bR^d_T)$ satisfying \eqref{eq0814_04}.
\end{theorem}

\section{Auxiliary results for equations with measurable coefficients}
\label{sec3}

In this section, as an intermediate step toward the proof of Theorem \ref{thm0808_01}, we show that the $L_p$-norm of $Du$ is controlled by the $L_p$-norms of $D_{x'}u = (D_2u, \ldots, D_du)$ and the right-hand side of the equation if $u \in \cH_{p,0}^{\alpha,1}(\bR^d_T)$ satisfies an equation as in \eqref{eq0814_06} and $a^{11}$ is a measurable function of either $t$ or $x_1$.
For this, in this section we first consider coefficients $a^{ij}$ satisfying Assumption \ref{assum0807_01} below, where $a^{ij}$ are measurable functions of either $t$ or $x_1$.
Note that Assumption \ref{assum0808_2} is strictly more general than Assumption \ref{assum0807_01} in that the coefficients $a^{ij}$, $(i,j) \neq (1,1)$, in the former assumption are allowed to be merely measurable in $(t,x_1)$.

\begin{assumption}
							\label{assum0807_01}
The coefficients $a^{ij}$ satisfy either (i) or (ii) of the following.
\begin{enumerate}
\item[(i)]
$a^{ij} = a^{ij}(t)$ for all $i,j=1,\ldots,d$.
\item[(ii)]
$a^{ij} = a^{ij}(x_1)$ for all $i,j=1,\ldots,d$.
\end{enumerate}
\end{assumption}

We first observe the following lemma.

\begin{lemma}
							\label{lem0711_1}
Let $\alpha \in (0,1)$, $p \in (1,\infty)$, $T \in (0,\infty)$, and $\Omega \subset \bR^d$.
For $k=1,\ldots,d$, $D_{x_k} u \in \cH_{p,0}^{\alpha,1}(\Omega_T)$ whenever $u \in \bH_{p,0}^{\alpha,2}(\Omega_T)$.
\end{lemma}

\begin{proof}
Denote $D_{x_k} u$ by $D_k u$.
To show that $D_k u \in \cH_{p,0}^{\alpha,1}(\Omega_T)$, we need to check that $D_k u \in \cH_p^{\alpha, 1}(\Omega_T)$ and there exists a sequence $\{v_n\} \subset C^\infty\left([0,T] \times \Omega\right)$ such that $v_n$ vanishes for large $|x|$, $v_n(0,x) = 0$, and
$$
\|v_n - D_k u\|_{\cH_p^{\alpha,1}(\Omega_T)} \to 0
$$
as $n \to \infty$.

Note that $D_k u, D D_k u \in L_p(\Omega_T)$.
Also note that, for any $\varphi \in C_0^\infty\left([0,T) \times \Omega\right)$,
$$
\int_0^T \int_{\Omega} I_0^{1-\alpha} (D_k u) \, \partial_t \varphi \, dx \, dt = \int_0^T \int_{\Omega} \partial_t \left(I_0^{1-\alpha} u\right) \, D_k \varphi \, dx \, dt.
$$
Thus, $D_k u \in \cH_p^{\alpha, 1}(\Omega_T)$ and
\begin{equation}
							\label{eq0711_01}
\partial_t I_0^{1-\alpha}(D_k u) = D_k \left( \partial_t I_0^{1-\alpha}u \right) \in \bH_p^{-1}(\Omega_T).
\end{equation}

To find a sequence $\{v_n\}$ satisfying the aforementioned properties, we recall that $u \in \bH_{p,0}^{\alpha,2}(\Omega_T)$, which means that there exists a sequence $\{u_n\} \subset C^\infty\left([0,T] \times \Omega\right)$ such that $u_n$ vanishes for large $|x|$, $u_n(0,x) = 0$, and
\begin{equation}
							\label{eq0710_04}
\|u_n - u\|_{\bH_p^{\alpha,2}(\Omega_T)} \to 0
\end{equation}
as $n \to \infty$.
Set $v_n = D_k u_n$.
Then, $\{v_n\} \subset C^\infty\left([0,T] \times \Omega\right)$, $v_n$ vanishes for large $|x|$, $v_n(0,x) = 0$, and
$$
\|v_n - D_k u\|_{L_p(\Omega_T)} + \|D v_n - D D_ku\|_{L_p(\Omega_T)} \to 0
$$
as $n \to \infty$.
We also have
$$
\| \partial_t I_0^{1-\alpha} v_n - \partial_t I_0^{1-\alpha} (D_k u)\|_{\bH_p^{-1}(\Omega_T)} \to 0
$$
as $n \to \infty$.
Indeed, \eqref{eq0711_01} shows that
$$
\partial_t \left(I_0^{1-\alpha} v_n - I_0^{1-\alpha} (D_k u) \right) = D_k \left( \partial_t I_0^{1-\alpha} u_n - \partial_t I_0^{1-\alpha} u \right).
$$
Hence,
$$
\| \partial_t I_0^{1-\alpha} v_n - \partial_t I_0^{1-\alpha} (D_k u)\|_{\bH_p^{-1}(\Omega_T)} \leq
\|\partial_t I_0^{1-\alpha} u_n - \partial_t I_0^{1-\alpha} u \|_{L_p(\Omega_T)} \to 0
$$
as $n \to \infty$, where the convergence of the right-hand side is guaranteed by \eqref{eq0710_04}.
The lemma is proved.
\end{proof}

We now prove an $L_p$-estimate for equations in divergence form when $a^{ij}$ are either functions of only $t$ or $x_1$.
In fact, the result follows from the previous results in \cite{MR3899965, MR4030286} combined with the above lemma.

\begin{proposition}[Divergence case]
							\label{prop0808_01}
Let $\alpha \in (0,1)$, $p \in (1,\infty)$, $T \in (0,\infty)$, and $a^{ij}$ satisfy Assumption \ref{assum0807_01}.
Then, for any $u \in \cH_{p,0}^{\alpha,1}(\bR^d_T)$ satisfying
\begin{equation}
							\label{eq0711_02}
- \partial_t^\alpha u + D_i \left( a^{ij} D_j u \right) = D_i g_i + f
\end{equation}
in $\bR^d_T$, where $g_i, f \in L_p(\bR^d_T)$, we have
\begin{equation}
							\label{eq0711_05}
\|Du\|_{L_p(\bR^d)} \leq N_0 \|g\|_{L_p(\bR^d_T)} + N_0T^{\alpha/2}\|f\|_{L_p(\bR^d_T)},
\end{equation}
where $N_0 = N_0(d,\delta,\alpha, p)$.
\end{proposition}

\begin{proof}
By scaling, without loss of generality we may assume that $T=1$.
In the case that $a^{ij} = a^{ij}(x_1)$, the result follows from \cite[Theorem 2.1]{MR4030286}.
In particular, see the estimate (2.4) in \cite[Theorem 2.1]{MR4030286}.

To deal with the case that $a^{ij}=a^{ij}(t)$, for each $k = 1,2,\ldots,d$, by using \cite[Theorem 2.1]{MR3899965} find a unique $w_k \in \bH_{p,0}^{\alpha,2}(\bR^d_1)$ satisfying
$$
- \partial_t^\alpha w_k + a^{ij}(t) D_{ij} w_k = g_k
$$
in $\bR^d_1$ with the estimate
\begin{equation}
							\label{eq0711_04}
\|D^2 w_k\|_{L_p(\bR^d_1)} \leq N_0 \|g_k\|_{L_p(\bR^d_1)},
\end{equation}
where $N_0 = N_0(d,\delta,\alpha,p)$.
There also exists a unique $w_0 \in \bH_{p,0}^{\alpha,2}(\bR^d_1)$ satisfying
$$
-\partial_t^\alpha w_0 + a^{ij}(t) D_{ij}w_0 = f
$$
in $\bR^d_1$ with the estimate
\begin{equation}
							\label{eq0713_01}
\|Dw_0\|_{L_p(\bR^d_1)} \leq N_1 \|f\|_{L_p(\bR^d_1)},
\end{equation}
where $N_1 = N_1(d,\delta,\alpha,p)$.
By Lemma \ref{lem0711_1}, $D_k w_k \in \cH_{p,0}^{\alpha,1}(\bR^d_1)$.
Clearly, $w_0 \in \cH_{p,0}^{\alpha,1}(\bR^d_1)$.
Moreover, $D_k w_k$ and $w_0$ satisfy the divergence type equations
\begin{equation*}
				%			\label{eq0711_03}
- \partial_t^\alpha D_k w_k + D_i\left(a^{ij}(t) D_j D_k w_k \right) = D_k g_k,
\end{equation*}
$$
- \partial_t^\alpha w_0 + D_i\left(a^{ij}(t) D_j w_0 \right) = f,
$$
respectively, in $\bR^d_1$.
Now we set
$$
v = w_0 + \sum_{k=1}^d D_k w_k,
$$
which belongs to $\cH_{p,0}^{\alpha,1}(\bR^d_1)$ and satisfies \eqref{eq0711_02}.
Then, if $p \in [2,\infty)$, by \cite[Lemma 4.1]{MR4030286} $v$ equals to $u$, which shows that the estimate \eqref{eq0711_05} follows from \eqref{eq0711_04} and \eqref{eq0713_01}.
For $p \in (1,2)$, we use the duality argument to obtain the estimate \eqref{eq0711_05}. See the proofs of \cite[Theorem 2.1]{MR4030286} and Proposition \ref{prop0118_1} below.
\end{proof}

We now present the main result of this section.
The proof is based on Proposition \ref{prop0808_01} and the scaling argument used, for instance, in \cite{MR2764911}.
Since the regularity of $a^{ij}$ for $(i,j) \neq (1,1)$ does not play any role, we allow them to be only measurable.

\begin{lemma}
							\label{lem0812_2}
Let $\alpha \in (0,1)$, $p \in (1,\infty)$, $T \in (0,\infty)$, and $a^{11}$ satisfy Assumption \ref{assum0808_2}.
The other coefficients $a^{ij}$, $(i,j) \neq (1,1)$, are measurable functions of $(t,x)$.
Then, for any $u \in \cH_{p,0}^{\alpha,1}(\bR^d_T)$ satisfying
\begin{equation*}
				%			\label{}
- \partial_t^\alpha u + D_i \left( a^{ij} D_j u \right) = D_i g_i + f
\end{equation*}
in $\bR^d_T$, where $g_i, f \in L_p(\bR^d_T)$, we have
\begin{equation}
							\label{eq0812_02}
\|Du\|_{L_p(\bR^d)} \leq N_0 \|D_{x'}u\|_{L_p(\bR^d)} + N_0 \|g\|_{L_p(\bR^d)} + N_0T^{\alpha/2} \|f\|_{L_p(\bR^d)},
\end{equation}
where $N_0 = N_0(d,\delta,\alpha, p)$.
\end{lemma}

\begin{proof}
Let $w(t,x_1,x') = u(\mu^{-2/\alpha}t,\mu^{-1}x_1,x')$ with an appropriate positive constant $\mu$ to be chosen below.
Note that $w$ belongs to $\cH_{p,0}^{\alpha,1}\left( (0,\mu^{2/\alpha}T) \times \bR^d \right)$ and satisfies
\begin{multline}
							\label{eq0812_01}
-\partial_t^\alpha w + D_1 \left( \tilde{a}^{11} D_1 w \right) + \mu^{-1} \sum_{i=2}^d D_i \left( \tilde{a}^{i1} D_1 w \right) + \mu^{-1} \sum_{j=2}^d D_1 \left( \tilde{a}^{1j} D_j w \right)
\\
+ \mu^{-2} \sum_{i,j=2}^d D_i \left( \tilde{a}^{ij} D_j w \right) = \mu^{-2} D_i \tilde{g}_i + \mu^{-2} \tilde{f}
\end{multline}
in $(0,\mu^{2/\alpha}T) \times \bR^d$, where
$$
\tilde{a}^{11} = a^{11}(\mu^{-2/\alpha} t) \quad \text{or} \quad \tilde{a}^{11} = a^{11}(\mu^{-1} x_1),
$$
$$
\tilde{a}^{ij}(t,x_1,x') = a^{ij}(\mu^{-2/\alpha}t, \mu^{-1}x_1,x'), \quad (i,j) \neq (1,1),
$$
$$
\tilde{g}(t,x_1,x) = (\mu g_1, g_2, \ldots, g_d)(\mu^{-2/\alpha}t,\mu^{-1}x_1,x'),
$$
$$
\tilde{f}(t,x_1,x') = f(\mu^{-2/\alpha}t,\mu^{-1}x_1,x').
$$
By adding $\Delta_{x'}w$ to both sides of \eqref{eq0812_01}, we write
\begin{align*}
&-\partial_t^\alpha w + D_1\left(\tilde{a}^{11} D_1 w \right) + \sum_{i=2}^d D_i^2 w = D_1 \bigg( \mu^{-2} \tilde{g}_1 - \mu^{-1} \sum_{j=2}^d \tilde{a}^{1j} D_j w \bigg)\\
&\quad+ \sum_{i=2}^d D_i \bigg(\mu^{-2}\tilde{g}_i + D_i w - \mu^{-1} \tilde{a}^{i1} D_1 w - \mu^{-2}\sum_{j=2}^d \tilde{a}^{ij} D_j w \bigg) + \mu^{-2}\tilde{f}
\end{align*}
in $(0,\mu^{2/\alpha}T) \times \bR^d$.
Since the coefficients of the above equation satisfy Assumption \ref{assum0807_01} (indeed, $\tilde{a}^{11}$ is a measurable function of only $t$ or $x_1$ and the other coefficients are constant), we apply Proposition \ref{prop0808_01} to get
\begin{align*}
\|Dw\|_p &\leq N_0 \left(\|D_{x'}w\|_p + \mu^{-1}\|Dw\|_p + \mu^{-2}\|D_{x'}w\|_p\right.\\
 &\quad \left.+ \mu^{-2}\|\tilde{g}_i\|_{L_p}\right) + N_0(\mu^{2/\alpha}T)^{\alpha/2} \mu^{-2}\|\tilde{f}\|_p,
\end{align*}
where $\|\cdot\|_p = \|\cdot\|_{L_p((0,\mu^{2/\alpha}T) \times \bR^d)}$ and $N_0 = N_0(d,\delta,\alpha,p)$.
We choose $\mu$ sufficiently large such that $N_0\mu^{-1} \leq 1/2$ and absorb the $\|Dw\|_p$ term to the left-hand side.
Finally, we return back to $u$ to obtain the estimate \eqref{eq0812_02}.
The lemma is proved.
\end{proof}

\section{Level set argument for equations in divergence form with measurable coefficients}
\label{sec4}

This section is devoted the level set estimate of $D_{x'}u$ for solutions to divergence form equations.

\begin{lemma}
							\label{lem0811_1}
Let $\alpha \in (0,1)$, $p_0\in (1,\infty)$, $- \infty < S \leq t_0 < T < \infty$, $0 < r < R < \infty$ and $a^{ij}$ satisfy Assumption \ref{assum0808_2}.
If Theorem \ref{thm0808_01} holds with this $p_0$ and $v\in \cH^{\alpha,1}_{p_0,0}\left((S,T)\times B_R\right)$ satisfies
\begin{equation*}
				%			\label{eq0710_03}
-\partial_t^\alpha v + D_i\left( a^{ij} D_j v \right) = D_i \Psi_i
\end{equation*}
in $(S,T) \times B_R$, where $\partial_t^\alpha = \partial_t I_S^{1-\alpha}$ and $\Psi_i \in L_{p_0}\left((S,T) \times B_R\right)$ with $\Psi_i = 0$ on $(t_0,T) \times B_R$,
then for any infinitely differentiable function $\eta(t)$ defined on $\bR$ such that $\eta(t) = 0$ for $t \leq t_0$, the function $D_\ell (\eta v)$, $\ell = 2, \ldots,d$, belongs to $\cH_{p_0,0}^{\alpha,1}\left((t_0,T) \times B_r\right)$ and satisfies
\begin{equation}
							\label{eq0810_01}
- \partial_t^\alpha \left( D_\ell (\eta v) \right) + D_i \left( a^{ij} D_j D_\ell (\eta v) \right) = \cG_\ell
\end{equation}
in $(t_0,T) \times B_r$,
where $\partial_t^\alpha = \partial_t I_{t_0}^{1-\alpha}$ and
\begin{equation*}
				%			\label{eq1116_03}
\cG_\ell = \frac{\alpha}{\Gamma(1-\alpha)} \int_S^t (t-s)^{-\alpha-1} \left(\eta(s) - \eta(t) \right) D_\ell v(s,x) \, ds.
\end{equation*}
Moreover, for $\ell = 2,\ldots,d$,
\begin{equation}
							\label{eq0812_03}
\|D_\ell(\eta v)\|_{\cH_{p_0}^{\alpha,1}\left((t_0,T) \times B_r\right)} \leq N \|D_\ell(\eta v)\|_{L_{p_0} \left((t_0,T) \times B_R\right)} + N \|\cG_\ell\|_{L_{p_0} \left((t_0,T) \times B_R\right)},
\end{equation}
where $N = N(d,\delta,\alpha,p_0,r,R)$.
If we additionally assume that $a^{ij}=a^{ij}(t)$, then $D_1(\eta v) \in \cH_{p_0,0}^{\alpha,1}\left((t_0,T) \times B_r\right)$ and $D_1(\eta v)$ satisfies \eqref{eq0810_01} and \eqref{eq0812_03} with $\ell = 1$.
\end{lemma}

\begin{proof}
Without loss of generality, we assume that $t_0 = 0$.
First, by Lemma 3.4 in \cite{MR4030286} $\eta v$ belongs to $\cH_{p_0,0}^{\alpha,1}\left((0,T) \times B_R\right)$ and satisfies
$$
\partial_t^\alpha (\eta v) = \partial_t I_0^{1-\alpha}(\eta v) = \eta \partial_t I_S^{1-\alpha} v - \cF
$$
in the sense of distribution in $(0,T) \times B_R$, where
\begin{equation*}
				%			\label{eq1113_11}
\cF(t,x) = \frac{\alpha}{\Gamma(1-\alpha)} \int_S^t (t-s)^{-\alpha-1} \left(\eta(s) - \eta(t) \right) v(s,x) \, ds.
\end{equation*}
That is, if
$$
\partial_t I_S^{1-\alpha} v = D_i g_i + f
$$
in $(S,T) \times B_R$ for $g_i, f \in L_{p_0}\left((S,T) \times B_R\right)$, then
\begin{equation}
							\label{eq0810_14}
\partial_t I_0^{1-\alpha} (\eta v) = D_i (\eta g_i) + \eta f - \cF
\end{equation}
in $(0,T) \times B_R$.
In particular, we have
\begin{equation}
							\label{eq0810_12}
- \partial_t^\alpha(\eta v) + D_i \left(a^{ij} D_j(\eta v) \right) = \cF
\end{equation}
in $(0,T) \times B_R$.

Next, we consider
$$
\Delta_{\ell,h}v(t,x) = \frac{v(t, x+ h e_\ell) - v(t,x)}{h}, \quad 0 < h < \frac{R-r}{2},
$$
which belongs to $\cH_{p_0,0}^1\left((0,T) \times B_{R_1}\right)$, where $R_1 = (R+r)/2$.
Since $v, Dv \in L_{p_0}\left((S,T) \times B_R\right)$, we have
\begin{equation}
							\label{eq0810_06}
\|\Delta_{\ell,h} v - D_\ell v\|_{L_{p_0}\left((S,T) \times B_{R_1}\right)} \to 0
\end{equation}
as $h \to 0$.
Upon noting that the coefficients $a^{ij}$ are functions of only $(t,x_1)$,
from \eqref{eq0810_12} we see that $\Delta_{\ell,h}v$, $\ell=2,\ldots,d$, satisfies
\begin{equation}
							\label{eq0810_02}
- \partial_t^\alpha(\eta \Delta_{\ell,h}v) + D_i \left(a^{ij} D_j(\eta \Delta_{\ell,h}v) \right) = \cG_{\ell,h}
\end{equation}
in $(0,T) \times B_{R_1}$, where
$$
\cG_{\ell,h} = \frac{\alpha}{\Gamma(1-\alpha)} \int_S^t (t-s)^{-\alpha-1} \left(\eta(s) - \eta(t) \right) \Delta_{\ell,h}v(s,x) \, ds.
$$
Note that $\cG_{\ell,h} \in L_{p_0}((0,T) \times B_{R_1})$ with
\begin{equation}
							\label{eq0810_03}
\|\cG_{\ell,h}\|_{L_{p_0}((0,T) \times B_{R_1})} \leq N \|\Delta_{\ell,h} v\|_{L_{p_0}((S,T) \times B_{R_1})} \leq N \|Dv\|_{L_{p_0}\left((S,T) \times B_R \right)},
\end{equation}
where $N$ depends only on $\alpha$, $p_0$, $S$, $T$, and the bound of $\eta'$.
See \cite[Lemma 3.4]{MR4030286}.
The first inequality in \eqref{eq0810_03} combined with \eqref{eq0810_06} shows that
\begin{equation}
							\label{eq0810_07}
\|\cG_{\ell,h} - \cG_\ell\|_{L_{p_0}((0,T) \times B_{R_1})} \le N\|\Delta_{\ell,h} v - D_\ell v\|_{L_{p_0}\left((S,T) \times B_{R_1}\right)} \to 0
\end{equation}
as $h \to 0$.
We now apply \cite[Lemma 4.3]{MR4030286} (this local estimate is available whenever the corresponding global estimate, i.e., Theorem \ref{thm0808_01} holds) to the equation \eqref{eq0810_02} to obtain
\begin{equation}
							\label{eq0810_04}
\begin{aligned}
\|D (\eta \Delta_{\ell,h} v)\|_{L_{p_0}\left((0,T) \times B_r\right)} \leq \frac{N}{R-r} \|\eta \Delta_{\ell,h} v\|_{L_{p_0}((0,T) \times B_{R_1})}
\\
+ N(R-r) \|\cG_{\ell,h}\|_{L_{p_0}((0,T) \times B_{R_1})} \leq N \|Dv\|_{L_{p_0}\left((S,T) \times B_R \right)},
\end{aligned}
\end{equation}
where $N = N(d,\delta,\alpha,p_0,r,R)$.
From this estimate with a slightly larger $r$, it follows that $D D_{x'} (\eta v) \in L_{p_0}\left((0,T) \times B_{r'}\right)$ for $r'\in (r, R_1)$, which combined with the property of $\Delta_{\ell ,h}$ implies
\begin{equation}
							\label{eq0810_05}
\|D(\eta \Delta_{\ell,h} v) - D D_\ell (\eta v)\|_{L_{p_0}\left((0,T) \times B_r\right)} \to 0
\end{equation}
as $h \to 0$.
Moreover, thanks to \eqref{eq0810_06} and  \eqref{eq0810_07}, by letting $h \to 0$ in the first inequality in \eqref{eq0810_04}, we have
\begin{equation}
							\label{eq0812_04}
\|DD_\ell(\eta v)\|_{L_{p_0}\left((0,T) \times B_r\right)} \leq N \|D_\ell(\eta v)\|_{L_{p_0} \left((0,T) \times B_R\right)} + N \|\cG_\ell\|_{L_{p_0} \left((0,T) \times B_R\right)},
\end{equation}
where $N = N(d,\delta,\alpha,p_0,r,R)$.
Now, since
\[
I_0^{1-\alpha} \left(\eta \Delta_{\ell,h} v \right) \to I_0^{1-\alpha} \left(\eta D_\ell v \right) \quad \text{in} \quad L_{p_0}\left((0,T) \times B_r\right)
\]
as $h \to 0$, by \eqref{eq0810_02} we see that $D_t^\alpha D_\ell (\eta v) \in \bH_{p_0}^{-1}\left((0,T) \times B_r\right)$ and
\begin{equation}
							\label{eq0810_10}
D_t^\alpha \left( D_\ell (\eta v) \right) = D_i \left( a^{ij} D_j ( D_\ell (\eta v) \right) - \cG_\ell, \quad \ell = 2,\ldots,d,
\end{equation}
in $(0,T) \times B_r$ with
\begin{equation}
							\label{eq0810_08}
\|D_t^\alpha \left( D_\ell (\eta v) \right) - \partial_t^\alpha \left(\eta \Delta_{\ell,h} v \right) \|_{\bH_{p_0}^{-1}\left((0,T) \times B_r\right)} \to 0
\end{equation}
as $h \to 0$.
In particular, the equality \eqref{eq0810_10} holds in the sense that
\begin{multline}
							\label{eq0810_11}
\int_0^T \int_{B_r} I_0^{1-\alpha} \left( D_\ell (\eta v) \right) \varphi_t \, dx \, dt
\\
= \int_0^T \int_{B_r} a^{ij} D_j\left(D_\ell (\eta v) \right) D_i \varphi \, dx \, dt + \int_0^T \int_{B_r} \cG_\ell \varphi \, dx \, dt
\end{multline}
for all $\varphi \in C_0^\infty\left([0,T) \times B_r\right)$.
Note that by the observations made above we have
$$
D_{x'} (\eta v), D D_{x'} (\eta v) \in L_{p_0}\left((0,T) \times B_r\right), \quad D_t^\alpha D_{x'} (\eta v) \in \bH_{p_0}^{-1}\left((0,T) \times B_r\right),
$$
from which and \eqref{eq0810_11}, however, one can only conclude that
$$
D_{x'} (\eta v) \in \cH_{p_0}^{\alpha,1}\left((0,T) \times B_r\right).
$$
Since $\cH_{p_0}^{\alpha,1}$ is a strict super set of $\cH_{p_0,0}^{\alpha,1}$ (see Section 3 in \cite{MR4030286}), to complete the proof,
we need to find a sequence $\{w_n\}$ such that $w_n \in C^\infty\left([0,T] \times B_r\right)$, $w_n(0,x) = 0$, and
\begin{align*}
\|w_n - D_\ell(\eta v)\|_{\cH_{p_0}^{\alpha,1}\left((0,T) \times B_r\right)} &= \| |w_n - D_\ell(\eta v)| + |Dw_n - DD_\ell(\eta v)| \|_{L_{p_0}(\left((0,T) \times B_r\right)}\\
&\quad + \|D_t^\alpha w_n - D_t^\alpha \left(D_\ell(\eta v) \right)\|_{\bH_{p_0}^{-1}\left((0,T) \times B_r\right)} \to 0
\end{align*}
as $n \to \infty$.
Assuming for the moment the existence of such a sequence and \eqref{eq0810_10}, one can conclude that $D_\ell(\eta v)$, $\ell = 2, \ldots,d$, belongs to $\cH_{p_0,0}^{\alpha,1}\left((0,T) \times B_r\right)$ and satisfies \eqref{eq0810_01}.
To see that \eqref{eq0812_03} holds, we simply combine \eqref{eq0812_04} and \eqref{eq0810_10} upon recalling the definition of the $\cH_{p_0}^{\alpha,1}$-norm, in particular, that of the $\bH_{p_0}^{-1}$-norm.

To find a sequence $\{w_n\}$, since $v \in \cH_{p_0,0}^{\alpha,1}\left((S,T) \times B_R\right)$, we recall that there exists a sequence $\{v_n\}$ such that $v_n \in C^\infty\left([S,T] \times B_R\right)$, $v_n(S,x) = 0$, and
$$
\|v_n - v\|_{\cH_{p_0}^{\alpha,1}\left((S,T) \times B_R\right)} \to 0
$$
as $n \to \infty$.
We then consider the sequence $\{\Delta_{\ell,h} (\eta v_n)\}_{h,n}$, which satisfies
$$
\Delta_{\ell,h} (\eta v_n) \in C^\infty\left([0,T] \times B_r\right),\quad
\Delta_{\ell,h}(\eta v_n)(0,x) = 0,
$$
and
$$
\| D_\ell (\eta v) - \Delta_{\ell,h}(\eta v_n)\|_{\cH_{p_0}^{\alpha,1}\left((0,T) \times B_r\right)} \to 0
$$
as $n \to \infty$ and $h \to 0$.
Indeed, the above quantity is bounded by
$$
\| D_\ell (\eta v) - \Delta_{\ell,h}(\eta v)\|_{\cH_{p_0}^{\alpha,1}\left((0,T) \times B_r\right)} + \| \Delta_{\ell,h}\left(\eta (v- v_n)\right)\|_{\cH_{p_0}^{\alpha,1}\left((0,T) \times B_r\right)},
$$
where the first term vanishes as $h \to 0$ due to \eqref{eq0810_06}, \eqref{eq0810_05}, and \eqref{eq0810_08}, and the second term can be made arbitrary small due to the choice of $\{v_n\}$ by letting $n \to \infty$ after a sufficiently small $h$ is fixed.
In particular, since $D_t^\alpha (v - v_n)$ vanishes as $n \to \infty$ in $\bH_{p_0}^{-1}\left((S,T) \times B_R\right)$, there exist $g_i^n, f^n \in L_{p_0}\left((S,T) \times B_R\right)$ such that $g_i^n, f^n \to 0$ in $L_{p_0}\left((S,T) \times B_R\right)$ and
$$
D_t^\alpha (v-v_n) = D_i g_i^n + f^n
$$
in $(S,T) \times B_R$.
Then, as seen in \eqref{eq0810_14},
\begin{equation}
							\label{eq0810_13}
D_t^\alpha \Delta_{\ell,h}\left(\eta (v- v_n)\right) = D_i \left( \Delta_{\ell,h}  (\eta g_i^n) \right) + \Delta_{\ell,h} (\eta f^n) - \cF_n,
\end{equation}
where
$$
\cF_n = \frac{\alpha}{\Gamma(1-\alpha)} \int_S^t (t-s)^{-\alpha-1} \left(\eta(s) - \eta(t) \right) \Delta_{\ell,h}(v-v_n) (s,x) \, ds,
$$
and, for each $h > 0$, the terms
$$
\Delta_{\ell,h}  (\eta g_i^n),\quad  \Delta_{\ell,h} (\eta f^n),\quad \text{and}\quad \cF_n
$$
on the right-hand side of \eqref{eq0810_13} vanish as $n \to 0$.
Therefore, one can come up with a desired sequence $\{w_n\}$ by choosing a subsequence $\{\Delta_{\ell,h_{k}} (\eta v_{n_k})\}$ with appropriate $h_{k} \to 0$ and $n_k\to \infty$.

If $a^{ij}=a^{ij}(t)$, we repeat the above proof with $\Delta_{1,h}v$. In particular, $\eta \Delta_{1,h}v$ satisfies \eqref{eq0810_02} with $\ell = 1$ because $a^{ij}$ are independent of $x_1$.
The lemma is proved.
\end{proof}

\begin{proposition}
                    \label{prop0811_1}
Let $\alpha \in (0,1)$, $p_0\in (1,\infty)$, $T \in (0,\infty)$, and $a^{ij}$ satisfy Assumption \ref{assum0807_01}.
If Theorem \ref{thm0808_01} holds with this $p_0$ and $u \in \cH^{\alpha,1}_{p_0,0}(\bR^d_T)$ satisfies
\begin{equation*}
				%			\label{eq0710_03b}
-\partial_t^\alpha u + D_i\left( a^{ij} D_j v \right) = D_i g_i
\end{equation*}
in $(0,T) \times \bR^d$, where $g_i \in L_{p_0}(\bR^d_T)$, then for any $(t_0,x_0) \in [0,T] \times \bR^d$ and $R \in (0,\infty)$, there exist
$$
w \in \cH_{p_0,0}^{\alpha,1} \left( (t_0 - R^{2/\alpha},t_0) \times \bR^d \right), \quad v \in \cH_{p_0,0}^{\alpha,1}\left((S,t_0) \times \bR^d \right),
$$
where $S:= \min\{0, t_0-R^{2/\alpha}\}$, satisfying the following.
\begin{align*}
{\rm (i)}& \,\,\, u = w + v \,\, \text{in} \,\, Q_R(t_0,x_0),
\\
{\rm (ii)}& \,\,\left(|Dw|^{p_0}\right)^{1/p_0}_{Q_R(t_0,x_0)} \leq N \left(|g|^{p_0}\right)^{1/p_0}_{Q_{2R}(t_0,x_0)},
\\
{\rm (iii)}& \,\,\left(|D_{x'}v|^{p_1}\right)^{1/p_1}_{Q_{R/2}(t_0,x_0)} \leq N \left(|g|^{p_0}\right)^{1/p_0}_{Q_{2R}(t_0,x_0)} + N \sum_{k=0}^\infty 2^{-k\alpha} \left( |D_{x'}u|^{p_0} \right)^{1/p_0}_{Q^k_R(t_0,x_0)},
\end{align*}
where $Q^k_R(t_0,x_0) = \left(t_0-(2^{k+1}+1)R^{2/\alpha}, t_0\right) \times B_R(x_0)$, $N = N(d,\delta,\alpha,p_0)$, and $p_1 = p_1(d, \alpha,p_0)\in (p_0,\infty]$ with
\begin{equation*}
				%			\label{eq0411_05}
\frac{1}{p_1} \geq \frac{1}{p_0} - \frac{1}{d+2/\alpha}.
\end{equation*}
Here we understand that $u$ and $g_i$ are extended to be zero for $t < 0$ and
$$
\left( |D_{x'}v|^{p_1} \right)_{Q_{r/2}(t_1,0)}^{1/p_1} = \|D_{x'} v\|_{L_\infty\left(Q_{r/2}(t_1,0)\right)}
$$
provided that $p_1 = \infty$.
\end{proposition}

\begin{proof}
The proof is almost the same as that of Proposition 5.1 in \cite{MR4030286}.
In particular, we use Lemma \ref{lem0811_1} and Corollaries \ref{cor1029_1} and \ref{cor0811_1} to establish the desired estimate for $D_{x'}v$.
See the corresponding estimate (5.7) in \cite{MR4030286}.
\end{proof}

Let $\gamma\in (0,1)$.
Also let $p_0 \in (1,\infty)$ and $p_1 \in (p_0,\infty]$ be from Proposition \ref{prop0811_1}.
Denote
\begin{equation*}
				%			\label{eq0406_04}
\cA(\textsf{s}) = \left\{ (t,x) \in (-\infty,T) \times \bR^d: |D_{x'} u(t,x)| > \textsf{s} \right\}
\end{equation*}
and
\begin{multline*}
				%			\label{eq0406_05}
\cB(\textsf{s}) = \big\{ (t,x) \in (-\infty,T) \times \bR^d:
\\
\gamma^{-1/p_0}\left( \cM |g|^{p_0} (t,x) \right)^{1/p_0} + \gamma^{-1/p_1}\left(\cS\cM |D_{x'} u|^{p_0}(t,x)\right)^{1/p_0} > \textsf{s}  \big\},
\end{multline*}
where we extend involved functions to be zero for $t \leq S$ if they are defined on $(S,T) \times \bR^d$.
Set
\begin{equation*}
				%			\label{eq0606_01}
\cC_R(t,x) = (t-R^{2/\alpha},t+R^{2/\alpha}) \times B_R(x), \quad \widehat{\cC}_R(t,x) = \cC_R(t,x) \cap \{t \leq T\}.
\end{equation*}

Since Proposition \ref{prop0811_1} is now available, by following the proof of Lemma 5.2 in \cite{MR4030286}, we obtain the following lemma.
Note that in the definition of $\cA(\mathsf{s})$ we have $D_{x'}u$, not $Du$.

\begin{lemma}
							\label{lem0812_1}
Let $\alpha \in (0,1)$, $p_0 \in (1,\infty)$, $T \in (0,\infty)$, and $a^{ij}$ satisfy Assumption \ref{assum0807_01}.
If Theorem \ref{thm0808_01} holds with this $p_0$ and $u \in \cH_{p_0,0}^{\alpha,1}(\bR^d_T)$ satisfies
$$
-\partial_t^\alpha u + D_i\left(a^{ij} D_j u \right) = D_i g_i
$$
in $\bR^d_T$, where $g_i \in L_{p_0}(\bR^d)$, then there exists a constant $\kappa = \kappa(d,\delta,\alpha,p_0) > 1$ such that the following holds: for $(t_0,x_0) \in (-\infty,T] \times \bR^d$ and $\mathsf{s} > 0$, if
$$
\left| \cC_{R/4}(t_0,x_0) \cap \cA(\kappa \mathsf{s})\right| \geq \gamma \left| \cC_{R/4}(t_0,x_0)\right|,
$$
then
$$
\widehat{\cC}_{R/4}(t_0,x_0) \subset \cB(\mathsf{s}).
$$
\end{lemma}

\section{Proofs of Theorems \ref{thm0808_01} and \ref{thm0808_02}}
\label{sec5}

In this section we present the proofs of Theorems \ref{thm0808_01} and \ref{thm0808_02}.

\begin{proof}[Proof of Theorem \ref{thm0808_01}]
As noted in Remark \ref{rem0811_1}, the theorem holds when $p=2$.
Suppose that the theorem holds for some $p_0 \in [2,\infty)$, which is indeed true for $p_0=2$.
Then, we fix $p_1 \in (p_0,\infty]$ from Proposition \ref{prop0811_1} and prove Theorem \ref{thm0808_01} for $p \in (p_0,p_1)$.
We first assume that $\lambda = 0$ and $f = 0$.
In this case by Lemma \ref{lem0812_1} and following the proof of Theorem 2.1 in \cite{MR4030286}, we obtain
$$
\|D_{x'}u\|_{L_p(\bR^d)} \leq N \|g\|_{L_p(\bR^d)},
$$
where $N = N(d,\delta,\alpha,p)$.
This estimate combined with Lemma \ref{lem0812_2} proves \eqref{eq0811_02} with $f=0$.
Then by using S. Agmon's idea (see \cite[Lemma 5.5]{MR2304157}), we see that, for $u \in \cH_{p,0}^{\alpha,1}(\bR^d_T)$ and $\lambda > 0$ satisfying \eqref{eq0814_06} in $\bR^d_T$, the estimate \eqref{eq0114_02} holds.
To prove \eqref{eq0811_02} with non-zero $f$ and $\lambda = 0$, we write the equation \eqref{eq0814_06} as
$$
-\partial_t^\alpha u + D_i\left(a^{ij}D_ju\right)- \lambda u = D_i g_i + f - \lambda u, \quad \lambda > 0,
$$
and apply the estimate \eqref{eq0114_02} to this equation to get
$$
\|Du\|_{L_p(\bR^d_T)} +\sqrt{\lambda}\|u\|_{L_p(\bR^d_T)} \leq N \|g\|_{L_p(\bR^d_T)} + \frac{N}{\sqrt{\lambda}} \|f\|_{L_p(\bR^d_T)} + N\sqrt{\lambda}\|u\|_{L_p(\bR^d_T)},
$$
where $N=N(d,\delta,\alpha,p)$.
By combining this estimate with Lemma 4.1 in \cite{MR4030286}, which is an estimate for the last term in the above inequality and does not require any regularity on the coefficients $a^{ij}$, we obtain
\begin{multline*}
\|Du\|_{L_p(\bR^d_T)} +\sqrt{\lambda}\|u\|_{L_p(\bR^d_T)} \leq N \|g\|_{L_p(\bR^d_T)} + \frac{N}{\sqrt{\lambda}} \|f\|_{L_p(\bR^d_T)}
\\
+ N\sqrt{\lambda}\left( T^{\alpha/2} \|g\|_{L_p(\bR^d_T)} + T^\alpha \|f\|_{L_p(\bR^d_T)} \right),
\end{multline*}
where $N = N(d,\delta,\alpha,p)$.
By choosing $\lambda = T^{-\alpha}$, we arrive at \eqref{eq0811_02} for non-zero $f$ and $\lambda = 0$.
To prove \eqref{eq0811_02} for $\lambda > 0$, we see that the estimates follow easily from \eqref{eq0114_02} when $\lambda \geq 1/(2N_0T^\alpha)$.
If $0 < \lambda < 1/(2N_0T^\alpha)$, we move $\lambda u$ to the right-hand side and use \eqref{eq0811_02} for $\lambda = 0$.
Because of the range of $\lambda$, the term involving $\|u\|_{L_p(\bR^d_T)}$ is absorbed to the left-hand side.

The existence assertion for $p \in (p_0,p_1)$ follows from the a priori estimates just proved above (also see Remark \ref{rem0115_1}) combined with the method of continuity and the existence result for equations having simple coefficients, for instance, in \cite[Theorem 2.1]{MR4030286}.

Now that we have proved Theorem \ref{thm0808_01} for $p \in [p_0,p_1)$, we repeat this procedure until we have $p_1 = \infty$ as in the proof of \cite[Theorem 2.1]{MR4030286} so that the theorem holds for all $p \in [2,\infty)$.

For $p \in (1,2)$, we use the duality argument as in the proof of \cite[Theorem 2.1]{MR4030286}.
The theorem is proved.
\end{proof}

We now prove Theorem \ref{thm0808_02} by following the proofs in \cite{MR2833589}, where the estimates for the non-divergence case are derived from those for the divergence case.

\begin{proof}[Proof of Theorem \ref{thm0808_02}]
We first prove the a priori estimate \eqref{eq0814_01} when $\lambda = 0$.
By Lemma \ref{lem0711_1}, we know that $D_1u \in \cH_{p,0}^{\alpha,1}(\bR^d_T)$.
Then one can check that $w := D_1u \in \cH_{p,0}^{\alpha,1}(\bR^d_T)$ satisfies the divergence type equation
$$
- \partial_t^\alpha  w + D_1\left(a^{11} D_1 w \right) + \Delta_{x'}w = D_1 \bigg(f - \sum_{(i,j)\neq (1,1)}^d a^{ij}D_{ij} u \bigg) + \sum_{i=2}^d D_i \left( D_i D_1 u \right)
$$
in $\bR^d_T$, where $\Delta_{x'} = \sum_{i=2}^d D_i^2$.
Since the coefficient matrix $\operatorname{diag}(a^{11},1,\ldots,1)$ of the above equation satisfies Assumption \ref{assum0808_2} with the same ellipticity constant $\delta$, by Theorem \ref{thm0808_01},  it follows that
\begin{equation}
							\label{eq0814_02}
\|DD_1u\|_{L_p(\bR^d_T)} = \|D w \|_{L_p(\bR^d_T)} \leq N \|f\|_{L_p(\bR^d_T)} + N \| DD_{x'}u \|_{L_p(\bR^d_T)},
\end{equation}
where $N = N(d,\delta,\alpha,p)$ and $DD_{x'}u = D_{ij}u$ with $(i,j) \neq (1,1)$.
To complete the proof, we prove
\begin{equation}
							\label{eq0814_03}
\|DD_{x'}u\|_{L_p(\bR^d_T)} \leq N \|f\|_{L_p(\bR^d_T)}.
\end{equation}
This combined with \eqref{eq0814_02} and the equation $\partial_t^\alpha u = a^{ij} D_{ij} u - f$ proves the estimate \eqref{eq0814_01}.
To prove \eqref{eq0814_03} we consider two cases depending on whether $a^{11} = a^{11}(t)$ or $a^{11} = a^{11}(x_1)$.

\noindent
{\bf Case 1}: $a^{11} = a^{11}(t)$.
In this case, the assumption $a^{11}=a^{11}(t)$ allows to write $a^{11}(t)D_1^2 u = D_1 \left(a^{11}(t) D_1 u \right)$, so that
the equation \eqref{eq0814_04} can be written as the following divergence type equation.
$$
- \partial_t^\alpha u + D_i \left( \tilde{a}^{ij} D_j u \right) = f,
$$
where
\begin{equation*}
				%			\label{eq1122_01}
\begin{aligned}
&\tilde{a}^{11} = a^{11}, \quad \tilde{a}^{ij} = a^{ij}, \quad i,j=2,\ldots,d,
\\
&\tilde{a}^{1j} = 0, \quad j = 2, \ldots, d, \quad \tilde{a}^{i1} = a^{1i} + a^{i1}, \quad i = 2,\ldots,d.
\end{aligned}
\end{equation*}
Note that the coefficients $\tilde{a}^{ij}$ satisfy Assumption \ref{assum0808_2} with the same ellipticity constant.
We then see that $w_\ell := D_\ell u$, $\ell = 2,\ldots,d$, belongs to $\cH_{p,0}^{\alpha,1}(\bR^d_T)$ and satisfies
$$
-\partial_t^\alpha w_\ell + D_i \left( \tilde{a}^{ij} D_j w_\ell \right) = D_\ell f
$$
in $\bR^d_T$.
By Theorem \ref{thm0808_01} again, where the symmetry of the coefficients, i.e., $\tilde{a}^{ij} = \tilde{a}^{ji}$, is not required, we obtain
$$
\|D D_{x'}u\|_{L_p(\bR^d_T)} = \|D w_\ell\|_{L_p(\bR^d_T)} \leq N \|f\|_{L_p(\bR^d_T)},
$$
where $N = N(d,\delta,\alpha,p)$, which proves \eqref{eq0814_03}.
Therefore, the estimate \eqref{eq0814_01} with $\lambda = 0$ is proved for the case $a^{11}=a^{11}(t)$.

\noindent
{\bf Case 2}: $a^{11} = a^{11}(x_1)$.
In this case, we make use of the following change of variables:
\begin{equation*}
				%			\label{eq1122_02}
y_1 = \chi(x_1) = \int_0^{x_1} \frac{1}{a^{11}(r)}\,dr, \quad y_i = x_i, \quad i = 2,\ldots,d.
\end{equation*}
As seen in the proof of Theorem 4.1 in \cite{MR2833589}, $v(t,y_1,y') = u(t,\chi^{-1}(y_1),y')$ satisfies
$$
- \partial_t^\alpha v + D_i \left( \hat{a}^{ij} D_j v \right) = \hat{f}
$$
in $\bR^d_T$, where
\begin{equation*}
				%			\label{eq1122_03}
\left\{
\begin{aligned}
&\hat{a}^{11}(y_1) = \frac{1}{a^{11}(\chi^{-1}(y_1))}, \quad \hat{a}^{1j} = 0, & j=2,\ldots,d,
\\
&\hat{a}^{i1}(t,y_1) = \hat{a}^{11}(y_1) \left(a^{1i}(t, \chi^{-1}(y_1)) + a^{i1}(t,\chi^{-1}(y_1))\right), & i=2,\ldots,d,
\\
&\hat{a}^{ij}(t,y_1) = a^{ij}(t,\chi^{-1}(y_1)), & i,j=2,\ldots,d,
\end{aligned}
\right.
\end{equation*}
$$
\hat{f}(t,y_1,y') = f(t,\chi^{-1}(y_1),y').
$$
Then, $D_\ell v$, $\ell = 2,\ldots,d$, belongs to $\cH_{p,0}^{\alpha,1}(\bR^d_T)$ and satisfies
\begin{equation}
							\label{eq0814_05}
-\partial_t^\alpha (D_\ell v) + D_i \left( \hat{a}^{ij} D_j (D_\ell v) \right) = D_\ell \hat{f}
\end{equation}
in $\bR^d_T$.
Note that $\hat{a}^{ij}$ satisfy Assumption \ref{assum0808_2} with an ellipticity constant depending only on $\delta$.
Indeed,
\begin{align*}
\hat{a}^{ij} \xi_i \xi_j &= \frac{1}{a^{11}} \left( \xi_1^2 + \sum_{j=2}^d a^{1j} \xi_i \xi_j + \sum_{i=2}^d a^{i1}\xi_i\xi_1 \right) + \sum_{i,j=2}^d a^{ij}\xi_i\xi_j\\
&= a^{ij} \tilde{\xi}_i\tilde{\xi}_j \geq \delta|\tilde{\xi}|^2 \geq \delta^3|\xi|^2, \quad \tilde{\xi} = \left( \frac{\xi_1}{a^{11}}, \xi_2, \ldots, \xi_d \right).
\end{align*}
Thus, Theorem \ref{thm0808_01} applied to \eqref{eq0814_05} shows
$$
\|D D_\ell v\|_{L_p(\bR^d_T)} \leq N \|\hat{f}\|_{L_p(\bR^d_T)},
$$
from which we obtain \eqref{eq0814_03}.
Therefore, the estimate \eqref{eq0814_01} with $\lambda = 0$ is also proved for the case $a^{11}=a^{11}(x_1)$.

Now that we have proved \eqref{eq0814_01} with $\lambda = 0$, using S. Agmon's idea as in the proof of Theorem \ref{thm0808_01}, we obtain the a priori estimate \eqref{eq0814_01} for $\lambda > 0$.
To prove \eqref{eq0115_01}, we write
\[
-\partial_t^\alpha u + a^{ij}D_{ij}u = f + \lambda u.
\]
By following the proof of \cite[Theorem 2.4]{MR3899965}, we have
\[
\|u\|_{L_p(\bR^d_T)} \leq N \|f\|_{L_p(\bR^d_T)} + \lambda N \|u\|_{L_p(\bR^d_T)},
\]
where $N = N(d,\delta,\alpha,p,T)$.
If $\lambda N \leq 1/2$, it follows that
\[
\|u\|_{L_p(\bR^d_T)} \leq 2N \|f\|_{L_p(\bR^d_T)}.
\]
From this estimate with \eqref{eq0814_01} as well as an interpolation inequality with respect to $x$ we arrive at \eqref{eq0115_01}.
If $\lambda \geq 1/(2N)$, the estimate \eqref{eq0115_01} follows directly from \eqref{eq0814_01}.

Finally, the existence assertion follows from the a priori estimates proved above combined with the method continuity and the solvability for equations with simple coefficients, for instance, in \cite{MR3899965}.
\end{proof}

\section{Boundary value problems}
\label{sec6}

We denote $\bR^d_+ = \{x = (x_1,x') \in \bR^d: x_1 > 0\}$.

\begin{theorem}[Half space]
							\label{thm0112_1}
The assertions in Theorems \ref{thm0808_01} and \ref{thm0808_02} hold if $\bR^d$ is replaced with $\bR^d_+$ and $u$ satisfies either the Dirichlet boundary condition (1) or the Neumann type boundary condition (2) below.

\begin{enumerate}
\item The Dirichlet boundary condition
\begin{equation}
							\label{eq0114_01}
u(t,0,x') = 0 \,\, \text{on} \,\,(0,T) \times \bR^{d-1}.
\end{equation}

\item
\begin{enumerate}
\item The divergence case: the conormal derivative boundary condition
\[
a^{1j}D_ju = g_1\,\, \text{on} \,\, (0,T) \times \bR^{d-1}.
\]

\item The non-divergence case: the Neumann boundary condition
\[
D_1u(t,0,x') = 0\,\, \text{on} \,\,(0,T) \times \bR^{d-1}.
\]
\end{enumerate}
\end{enumerate}
\end{theorem}

\begin{proof}
We use even or odd extensions with respect to $x_1$.
See, for instance, the proofs of \cite[Theorem 2.4]{MR2764911} and \cite[Theorem 2.7]{MR2300337}.
In particular, if $u \in \cH_{p,0}^{\alpha,1}(\Omega_T)$ or $u \in \bH_{p,0}^{\alpha,2}(\Omega_T)$, where $\Omega = \bR^d_+$, with the boundary condition \eqref{eq0114_01}, we see that $\bar{u}$, the odd extension of $u$ with respect to $x_1$ given by
\[
\bar{u} =
\left\{\begin{aligned}
u(t,x_1,x') \quad &\text{for} \quad x_1 > 0,
\\
-u(t,-x_1,x') \quad &\text{for} \quad x_1 < 0,
\end{aligned}
\right.
\]
belongs to $\cH_{p,0}^{\alpha,1}(\bR^d_T)$ or $\bH_{p,0}^{\alpha,2}(\bR^d_T)$, respectively.
We have the same conclusion for the even extension of $u$ with respect to $x_1$ when the conormal derivative boundary condition or the Neumann boundary condition is imposed.
\end{proof}

\begin{proposition}[Partially bounded domain]
							\label{prop0118_1}
Let $\alpha \in (0,1)$, $p \in (1,\infty)$, $T \in (0,\infty)$, $R \in (0,\infty)$, and
\[
\Pi = \{(x_1,x') \in \bR^d: 0<x_1<R, x' \in \bR^{d-1}\}.
\]
Under Assumption \ref{assum0808_2}, for any $u \in \cH_{p,0}^{\alpha,1}(\Pi_T)$ and $\lambda \geq 0$ satisfying \eqref{eq0814_06} in $\Pi_T$, where $g_i, f \in L_p(\Pi_T)$, with the Dirichlet boundary condition
\begin{equation}
							\label{eq0118_02}
u(t,x) = 0 \,\, \text{on} \,\, (0,T) \times \partial \Pi,
\end{equation}
we have the estimates \eqref{eq0114_02} when $\lambda>0$ and \eqref{eq0811_02} when $\lambda\ge 0$, with $\bR^d_T$ replaced by $\Pi_T$, where $N_0$ again depends only on $d$, $\delta$, $\alpha$, and $p$.
Moreover, for any $g_i, f \in L_p(\Pi_T)$ and $\lambda\ge 0$, there exists a unique $u \in \cH_{p,0}^{\alpha,1}(\Pi_T)$ satisfying \eqref{eq0814_06} in $\Pi_T$ with the boundary condition \eqref{eq0118_02}.

Similar assertions hold for the divergence type equation \eqref{eq0814_06} with the conormal derivative boundary condition and the non-divergence type equation \eqref{eq0814_04} with either the Dirichlet boundary condition or the Neumann boundary condition.
\end{proposition}

\begin{proof}
Since the coefficients are measurable functions of $t$, $x_1$, or $(t,x_1)$, and the desired estimates as in \eqref{eq0114_02} and \eqref{eq0814_01} (the constant $N_0$ in \eqref{eq0811_02} as well) are independent of $T$, one can use a scaling argument.
Thus, we may assume $R=1$.

We first prove the a priori estimates.
We only give details for the divergence case with the Dirichlet boundary condition. The other cases are proved similarly.

Let $u \in \cH_{p,0}^{\alpha,1}(\Pi_T)$ satisfy \eqref{eq0814_06} with the Dirichlet boundary condition \eqref{eq0118_02}.
We first take the odd extension of $u$ with respect to $\{x_1=0\}$ and then take the periodic extension of it to the whole space $\bR^d$, so that the extended function, which is still denoted by $u$, is in $\cH_{p,\text{loc}}^{\alpha,1}(\bR^d_T)$. Similarly, we take the odd (and even) extensions of $f$ and $g_i,i=2,\ldots,d$ (and $g_1$) with respect to $\{x_1=0\}$ and then take the periodic extensions of them to the whole space.
For the coefficients $a^{1j}$ and $a^{j1}$, $j=2,\ldots,d$, we take the periodic extensions (again denoted by $a^{1j}$) of the oddly extended ones with respect to $\{x_1=0\}$.
For the remaining coefficients, we take the periodic extensions of the evenly extended ones with respect to $\{x_1=0\}$.
The extended coefficients satisfy Assumption \ref{assum0808_2} as well as the uniform ellipticity condition.
Next let $\eta\in C_0^\infty(B_2)$ be such that $\eta=1$ on $B_1$ and $\eta_R(\cdot)=\eta(\cdot/R)$ for $R\ge 10$. Then it is easily seen that $u\eta_R\in \cH_{p,0}^{\alpha,1}(\bR^d_T)$ satisfies
\begin{equation*}
\begin{aligned}
-\partial_t^\alpha(u\eta_R) + D_i\left(a^{ij} D_j (u\eta_R) \right) - \lambda u\eta_R &= D_i (\eta_R g_i) + \eta_R f
\\
+ D_i\left( a^{ij} (D_j \eta_R) u\right) &+ a^{ij} D_i\eta_R D_j u - (D_i \eta_R) g_i
\end{aligned}
\end{equation*}
in $\bR^d_T$.
We first prove \eqref{eq0114_02} with $\bR^d_T$ replaced with $\Pi_T$.
By applying Theorem \ref{thm0808_01} to the above equation, we get
\begin{align*}
&\|D(\eta_R u)\|_{L_p(\bR^d_T)} + \sqrt{\lambda}\|\eta_R u\|_{L_p(\bR^d_T)} \leq N_0 \||\eta_R g| + |(D\eta_R) u| \|_{L_p(\bR^d_T)}\\
&\qquad + \frac{N_0}{\sqrt{\lambda}}\||\eta_R f| + |D\eta_R Du| + |(D\eta_R)||g|\|_{L_p(\bR^d_T)},
\end{align*}
where $N_0 = N_0(d,\delta,\alpha,p)$. This together with the periodicity of $u$, $f$, and $g_i$ in $x_1$ and the bound $|D\eta_R|\le NR^{-1}$ gives
\begin{align*}
&\||Du|+\sqrt\lambda |u|\|_{L_p((0,T)\times (0,1)\times B_{R/2}')}
\leq N_0 \||g| + R^{-1}|u| \|_{L_p((0,T)\times (0,1)\times B_{2R}')}\\
&\qquad + \frac{N_0}{\sqrt{\lambda}}\||f| + R^{-1}|Du| + R^{-1}|g|\|_{L_p((0,T)\times (0,1)\times B_{2R}')},
\end{align*}
where $B_r' = \{x'\in \bR^{d-1}: |x'|<r\}$.
Taking the limits as $R\to \infty$ gives the desired estimate
\begin{equation*}
                %                \label{eq5.04}
\|Du\|_{L_p(\Pi_T)}+ \sqrt{\lambda}\|u\|_{L_p(\Pi_T)} \leq N_0 \|g\|_{L_p(\Pi_T)} + \frac{N_0}{\sqrt{\lambda}}\|f\|_{L_p(\Pi_T)},
\end{equation*}
where $N_0=N_0(d,\delta,\alpha,p)$.
Similarly, we obtain \eqref{eq0811_02} from Theorem \ref{thm0808_01}.

We now prove the existence of  unique solutions to the equations.
Thanks to the a priori estimates just proved above, the uniqueness is guaranteed whenever solutions exist.
Thus, we only prove the existence of solutions, for which by the a priori estimates and the method of continuity, in the divergence case it is enough to show the existence of a solution $u \in \cH_{p,0}^{\alpha,1}(\Pi_T)$ to the equation
\begin{equation}
							\label{eq0120_01}
-\partial_t^\alpha u + \Delta u  = D_ig_i+f
\end{equation}
in $\Pi_T$, where $g_i,f \in L_p(\Pi_T)$, with either the Dirichlet boundary condition or the conormal derivative boundary condition.
Indeed, even if equations have simple coefficients as in \eqref{eq0120_01}, the corresponding solvability result does not seem available in the literature when the spatial domain is $\Pi$.
We give here details for the case with the Dirichlet boundary condition \eqref{eq0118_02}.

By using the density argument and the a priori estimates, without loss of generality we may assume that $g_i$ and $f$ are smooth with compact support in $\Pi_T$.
In the proof below whenever $g_i$ and $f$ are considered in $\bR^d_T$, those are the extended ones as in the proof of the a priori estimates.

{\bf Case 1}: $p=2$. In this case, the solvability follows from the Galerkin method. See \cite{MR2538276}.

{\bf Case 2}: $p\in (1,2)$. From Step 1, we know that \eqref{eq0120_01} has a solution $u\in \cH_{2,0}^{\alpha,1}(\Pi_T)$. This solution is also in $\cH_{p,0}^{\alpha,1}(\Pi_T)$. Indeed, we take the extension of $u$ as in the proof of the a priori estimates. Then it is easily seen that
$u\eta_R\in \cH_{p,0}^{\alpha,1}(\bR^d_T)\cap \cH_{2,0}^{\alpha,1}(\bR^d_T)$ satisfies
\begin{equation*}
\begin{aligned}
-\partial_t^\alpha(u\eta_R) + \Delta(u\eta_R)&= D_i (\eta_R g_i) + \eta_R f
\\
+ 2D_i\left( (D_i \eta_R) u\right) &- (\Delta\eta_R) u - (D_i \eta_R) g_i
\end{aligned}
\end{equation*}
in $\bR^d_T$.
By applying Theorem \ref{thm0808_01} to the above equation we get
\begin{align*}
             %       \label{eq5.55}
T^{\alpha/2}\|D(\eta_R u)\|_{L_p(\bR^d_T)} &+ \|\eta_R u\|_{L_p(\bR^d_T)} \leq
N_0 T^{\alpha/2}\||\eta_R g| + |(D\eta_R) u| \|_{L_p(\bR^d_T)}\notag\\
&\quad + N_0T^{\alpha}\||\eta_R f| + |(\Delta\eta_R) u| + |(D\eta_R)||g|\|_{L_p(\bR^d_T)},
\end{align*}
where $N_0 = N_0(d,\delta,\alpha,p)$. This together with the periodicity of $u$, $f$, and $g_i$ in $x_1$ and the bounds $|D\eta_R|\le NR^{-1},|\Delta\eta_R|\le NR^{-2}$ gives
\begin{align*}
&T^{\alpha/2}\|Du\|_{L_p((0,T)\times (0,1)\times B_{R/2}')} + \|u\|_{L_p((0,T)\times (0,1)\times B_{R/2}')}\\
&\leq N_0T^{\alpha/2} \||g| + R^{-1}|u| \|_{L_p((0,T)\times (0,1)\times B_{2R}')}\\
&\quad +N_0T^{\alpha}\||f| + R^{-2}|u| + R^{-1}|g|\|_{L_p((0,T)\times (0,1)\times B_{2R}')}.
\end{align*}
By H\"older's inequality,
\begin{align*}
&R^{-1}\|u\|_{L_p((0,T)\times (0,1)\times B_{2R}')}\\
&\le NT^{1/p - 1/2} R^{-1+(d-1)(1/p-1/2)}\|u\|_{L_2((0,T)\times (0,1)\times B_{2R}')}.
\end{align*}
Combining the above two inequalities and sending $R\to \infty$, we see that $u,Du\in L_p(\Pi_T)$ provided that $1/p-1/2<1/(d-1)$.
In the general case, we take a sequence of decreasing exponents $\{p_i\}_{i=0}^m\subset (1,2)$ such that $p_0=2$, $p_m=p$, and $1/p_{i+1}-1/p_{i}<1/(d-1)$ for $i=0,\ldots,m-1$.
By the proof above with $p=p_1$, we have $u,Du\in L_{p_1}(\Pi_T)$. With $p_1$ and $p_2$ in place of $2$ and $p$ in the proof, we further deduce $u,Du\in L_{p_2}(\Pi_T)$. Repeating this procedure gives $u,Du\in L_{p}(\Pi_T)$.
Finally, by using the equation, we have $u\in \cH_{p,0}^{\alpha,1}(\Pi_T)$.

%=======================================
{\bf Case 3}: $p\in (2,\infty)$. As in Case 2, \eqref{eq0120_01} has a solution $u\in \cH_{2,0}^{\alpha,1}(\Pi_T)$. Next, we show that $u\in \cH_{p,0}^{\alpha,1}(\Pi_T)$ by using a duality argument.
Let $q=p/(p-1)\in (1,2)$ and $\tilde g_i,\tilde f\in C_0^\infty(\Pi_T)$.
By the result proved in Case 2, we know that the equation
$$
-\partial_t^\alpha v(t,x) + \Delta v(t,x) = D_i\tilde g_i(T-t,x)+\tilde f(T-t,x)
$$
in $\Pi_T$ has a unique solution $v\in \cH_{q,0}^{\alpha,1}(\Pi_T)\cap \cH_{2,0}^{\alpha,1}(\Pi_T)$ and
\begin{equation}
                            \label{eq6.09}
\|Dv\|_{L_q(\Pi_T)}+\|v\|_{L_q(\Pi_T)}\le N\|\tilde g\|_{L_q(\Pi_T)}+N\|\tilde{f}\|_{L_q(\Pi_T)},
\end{equation}
where $N=N(d,\delta,\alpha,q,T)$.
Let $w(t,x)=v(T-t,x)$.
By duality, it is easily seen that
$$
\int_{\Pi_T}(D_iu \tilde g_i-u\tilde f)\,dxdt=\int_{\Pi_T}(D_i w g_i-w f)\,dxdt.
$$
Therefore, by H\"older's inequality and \eqref{eq6.09},
\begin{align*}
\Big|\int_{\Pi_T}(D_iu \tilde g_i-u\tilde f)\,dxdt\Big|
&\le N\|Dw\|_{L_q(\Pi_T)}\|g\|_{L_p(\Pi_T)}+N\|w\|_{L_q(\Pi_T)}\|f\|_{L_p(\Pi_T)}\\
&= N\|Dv\|_{L_q(\Pi_T)}\|g\|_{L_p(\Pi_T)}+N\|v\|_{L_q(\Pi_T)}\|f\|_{L_p(\Pi_T)}\\
&\le N\Big(\|\tilde g\|_{L_q(\Pi_T)}+\|\tilde f\|_{L_q(\Pi_T)}\Big)\Big(\|g\|_{L_p(\Pi_T)}+\|f\|_{L_p(\Pi_T)}\Big).
\end{align*}
Since $\tilde g_i, \tilde{f}\in C_0^\infty(\Pi_T)$ are arbitrary and $C_0^\infty(\Pi_T)$ is dense in $L_q(\Pi_T)$, we see that $u,Du\in L_p(\Pi_T)$. It then follows from the equation that $u\in \cH_{p,0}^{\alpha,1}(\Pi_T)$.

For non-divergence form equations, we need to show the existence of a solution $u \in \bH_{p,0}^{\alpha,2}(\Pi_T)$ to the equation
\begin{equation}
							\label{eq0120_01b}
-\partial_t^\alpha u + \Delta u  = f
\end{equation}
in $\Pi_T$, where $f \in L_p(\Pi_T)$, with either the Dirichlet boundary condition or the Neumann boundary condition. By the solvability of divergence form equations obtained above, \eqref{eq0120_01b} has a unique solution $u\in \cH_{p,0}^{\alpha,1}(\Pi_T)$. To show that $u\in \bH_{p,0}^{\alpha,2}(\Pi_T)$, we again use the extension argument. Clearly, $u\eta_R\in \cH_{p,0}^{\alpha,1}(\bR^d_T)$ satisfies
\begin{equation}
                        \label{eq2.37}
-\partial_t^\alpha(u\eta_R) + \Delta(u\eta_R)= \eta_R f
+ 2 (D_i \eta_R) (D_i u)+(\Delta\eta_R) u
\end{equation}
in $\bR^d_T$, where $u$ and $f$ are extended ones as above.
Since the right-hand side of \eqref{eq2.37} is in $L_p(\bR^d_T)$, by Theorem \ref{thm0808_02}, there is a unique solution $\tilde u\in \bH_{p,0}^{\alpha,2}(\bR^d_T)$ to \eqref{eq2.37}. By the inclusion $\bH_{p,0}^{\alpha,2}(\bR^d_T)\subset \cH_{p,0}^{\alpha,1}(\bR^d_T)$ and the uniqueness of $\cH_{p,0}^{\alpha,1}(\bR^d_T)$-solutions, we have $\tilde u=u\eta_R$. Thus, by Theorem \ref{thm0808_02} again, it holds that
$$
\|u\eta_R\|_{\bH_p^{\alpha,2}(\bR^d_T)} \leq N\|\eta_R f
+ 2 (D_i \eta_R) (D_i u)+(\Delta\eta_R) u\|_{L_p(\bR^d_T)}.
$$
As before, by sending $R\to \infty$, we reach
$$
\|u\|_{\bH_p^{\alpha,2}(\Pi_T)} \leq N\|f\|_{L_p(\Pi_T)}.
$$
The proposition is proved.
\end{proof}

\section{Optimal embeddings for \texorpdfstring{$\cH_{p,0}^{\alpha,1}$}{}}

In this section we present elementary and self-contained proofs of H\"{o}lder embeddings (Theorem \ref{thm0717_1}) and Sobolev embeddings (Theorem \ref{thm0811_01}) for fractional parabolic Sobolev spaces $\cH_{p,0}^{\alpha,1}$.

We begin with the following two lemmas.

\begin{lemma}
							\label{lem0716_1}
Let $\alpha \in (0,1)$, $T \in (0,\infty)$, and $u, \eta
\in C^1\left([0,T]\right)$ with
$u(0)=\eta(T) = 0$.
Then
\begin{equation}
							\label{eq0716_01}
\int_0^T u'(t) \eta(t) \, dt = - \int_0^T \partial_t I_0^{1-\alpha} u (t) J_T^\alpha \eta'(t)\, dt,
\end{equation}
where
$$
J_T^{\alpha} v (t) = \frac{1}{\Gamma(\alpha)} \int_t^T (s-t)^{\alpha - 1} v(s) \, ds.
$$
If $J_T^\alpha \eta'(t) \in C^1\left([0,T]\right)$, we further have
\begin{equation}
							\label{eq0716_02}
\int_0^T u'(t) \eta(t) \, dt = \int_0^T I_0^{1-\alpha} u (t) \, \partial_t \left(J_T^\alpha \eta'(t) \right)\, dt.
\end{equation}
\end{lemma}

\begin{proof}
Since $u(0) = 0$, by Lemma A.4 in \cite{MR3899965} we have
$$
I_0^\alpha D_t^\alpha u(t) = I_0^\alpha \partial_t^\alpha u(t) = u(t).
$$
Hence,
\begin{align*}
&\int_0^T u'(t) \eta(t) \, dt = - \int_0^T u(t) \eta'(t) \, dt = - \int_0^T I_0^\alpha D_t^\alpha u (t) \, \eta'(t) \, dt
\\
&= - \frac{1}{\Gamma(\alpha)} \int_0^T \int_0^t (t-s)^{\alpha-1} D_t^\alpha u(s) \, ds \, \eta'(t) \, dt
\\
&= - \frac{1}{\Gamma(\alpha)} \int_0^T D_t^\alpha u(s) \int_s^T (t-s)^{\alpha-1} \eta'(t) \, dt \, ds
= - \int_0^T \partial_t I_0^{1-\alpha} u(s) J_T^{\alpha}\eta'(s) \, ds,
\end{align*}
which gives \eqref{eq0716_01}.
If $J_T^\alpha \eta'(t) \in C^1\left([0,T]\right)$, using integration by parts with
$$
I_0^{1-\alpha}u(s)|_{s=0} = J_T^\alpha \eta'(s)|_{s=T} = 0,
$$
we arrive at \eqref{eq0716_02}.
\end{proof}

\begin{lemma}
							\label{lem0717_01}
Let $T \in (0,\infty)$, $\alpha \in (0,1)$, and $p \in (1,\infty)$ such that
$$
\sigma := 1 - (d+2/\alpha)/p > 0.
$$
Also let $u \in C_0^\infty\left([0,T] \times \bR^d\right)$ satisfy $u(0,x) = 0$ and
\begin{equation}
							\label{eq0731_01}
\int_0^T \int_{\bR^d} I_0^{1-\alpha} u \varphi_t \, dx \, dt = \int_0^T \int_{\bR^d} g_i D_i \varphi \, dx \, dt - \int_0^T \int_{\bR^d} f \varphi \, dx \, dt
\end{equation}
for any $\varphi \in C_0^\infty\left([0,T) \times \bR^d\right)$, where $g_i, f \in L_p(\bR^d_T)$.
Then, for $t_1, t_2 \in [0,T]$ with $t_1 < t_2$ and $r > 0$, we have
\begin{multline}
							\label{eq0717_02}
\left|\int_{B_r(x_1)} \left( u(t_2,x) - u(t_1,x) \right) \phi(x) \, dx \right|
\\
\leq N (t_2-t_1)^{\alpha-1/p} \left(\|g\|_{L_p(\bR^d_T)} + \|f\|_{L_p(\bR^d_T)}\right)\|\phi\|_{W_{p'}^1(B_r(x_1))},
\end{multline}
where $\phi \in C_0^\infty\left(B_r(x_1)\right)$,
$N = N(\alpha,p)$, and $1/p + 1/p' = 1$.
\end{lemma}

\begin{proof}
Without loss of generality, we assume that $x_1 = 0$.
Let $\zeta(s)$ be an infinitely differentiable function such that $\zeta(s) \geq 0$,
$\zeta(s) = 0$ for $s \geq 0$, and
$$
\int_\bR \zeta(s) \, ds = \int_{-\infty}^0 \zeta(s) \, ds = 1.
$$
Set $\zeta_\varepsilon(s) = \varepsilon^{-1}\zeta(s/\varepsilon)$ and
$$
\eta_\varepsilon(t) = \int_{-\infty}^\infty 1_{(t_1,t_2)}(s) \zeta_\varepsilon(t-s) \, ds.
$$
Note that $\eta_\varepsilon$ is an approximation of $1_{(t_1,t_2)}$ and $\eta_\varepsilon(t) = 0$ for $t \geq t_2$.
We then write
\begin{align*}
&u(t_2,x) - u(t_1,x) = \int_{t_1}^{t_2} \partial_t u(t,x) \, dt = \int_0^{t_2} \partial_t u(t,x) 1_{(t_1,t_2)}(t) \, dt\\
&= \lim_{\varepsilon \to 0} \int_0^{t_2} \partial_t u(t,x) \eta_\varepsilon(t) \, dt = \lim_{\varepsilon \to 0} \int_0^{t_2} I_0^{1-\alpha}u(t,x) \partial_t \left[ J_{t_2}^\alpha \eta_\varepsilon'(t) \right] \, dt,
\end{align*}
where the last equality follows from Lemma \ref{lem0716_1} with $t_2$ in place of $T$ along with the observations that $u(0,x) = \eta_\varepsilon(t_2) = 0$ and $J_{t_2}^\alpha \eta_\varepsilon'(t) \in C^1([0,T])$.
Then, for $\phi \in C_0^\infty(B_r)$,
\begin{align*}
&\int_{B_r} (u(t_2,x) - u(t_1,x)) \phi(x) \, dx
= \lim_{\varepsilon \to 0} \int_{\bR^d} \int_0^{t_2} \partial_t u(t,x) \eta_\varepsilon(t) \, dt \, \phi(x) \, dx
\\
&= \lim_{\varepsilon \to 0} \int_0^{t_2} \int_{\bR^d}I_0^{1-\alpha}u(t,x) \partial_t\left[J_{t_2}^\alpha \eta'_\varepsilon(t) \phi(x)\right] \, dx \, dt
\\
&= \lim_{\varepsilon \to 0} \left(\int_0^{t_2} \int_{\bR^d} g_i(t,x) J_{t_2}^\alpha \eta'_\varepsilon(t) D_i \phi(x) \, dx \, dt - \int_0^{t_2} \int_{\bR^d} f(t,x) J_{t_2}^\alpha \eta'_\varepsilon(t) \phi(x) \, dx \, dt\right),
\end{align*}
where the last equality is due to \eqref{eq0731_01}, which is satisfied with $t_2$ in place $T$ since $J_{t_2}^\alpha \eta'_\varepsilon(t) \phi(x) \in C_0^\infty\left([0,t_2) \times \bR^d\right)$ and $J_{t_2}^\alpha \eta'_\varepsilon(t) \phi(x) = 0$ for $t \geq t_2$.
This shows that
\begin{multline}
							\label{eq0717_03}
\left|\int_{B_r} \left(u(t_2,x) - u(t_1,x) \right) \phi(x) \, dx \right| \leq \limsup_{\varepsilon \to 0} \left|\int_0^{t_2} \int_{\bR^d} f(t,x) J_{t_2}^\alpha \eta'_\varepsilon(t) \phi(x) \, dx \, dt\right|
\\
+ \limsup_{\varepsilon \to 0} \left|\int_0^{t_2} \int_{\bR^d} g_i(t,x) J_{t_2}^\alpha \eta'_\varepsilon(t) D_i \phi(x) \, dx \, dt\right|,
\end{multline}
where, for each $\varepsilon > 0$,
\begin{multline}
							\label{eq0717_04}	
\left|\int_0^{t_2} \int_{\bR^d} g_i(t,x) J_{t_2}^\alpha \eta'_\varepsilon(t) D_i \phi(x) \, dx \, dt\right|
\\
\leq \|g\|_{L_p(\bR^d_T)} \|D\phi\|_{L_{p'}(B_r)} \left(\int_0^{t_2} \left| J_{t_2}^\alpha \eta_\varepsilon'(t)\right|^{p'} \, dt \right)^{1/p'}
\end{multline}
and
\begin{multline}
							\label{eq0717_05}
\left|\int_0^{t_2} \int_{\bR^d} f(t,x) J_{t_2}^\alpha \eta'_\varepsilon(t) \phi(x) \, dx \, dt\right|
\\
\leq \|f\|_{L_p(\bR^d_T)} \|\phi\|_{L_{p'}(B_r)} \left(\int_0^{t_2} \left| J_{t_2}^\alpha \eta_\varepsilon'(t)\right|^{p'} \, dt \right)^{1/p'}.
\end{multline}
We now prove that
\begin{equation}
							\label{eq0717_06}
\left(\int_0^{t_2} \left| J_{t_2}^\alpha \eta_\varepsilon'(t)\right|^{p'} \, dt \right)^{1/p'} \leq N(\alpha,p)(t_2-t_1)^{\alpha-1/p}.
\end{equation}
If this holds, we obtain \eqref{eq0717_02} by combining \eqref{eq0717_03}, \eqref{eq0717_04}, \eqref{eq0717_05}, and \eqref{eq0717_06}.
To prove \eqref{eq0717_06}, we note that
$$
\eta_\varepsilon(t) = \int_{t_1}^{t_2} \zeta_{\varepsilon}(t-s) \, ds = \int_{t-t_2}^{t-t_1} \zeta_\varepsilon(s) \, ds.
$$
This gives
$$
\eta_\varepsilon'(t) = \zeta_\varepsilon(t-t_1) - \zeta_\varepsilon(t-t_2)
$$
and
\begin{align*}
&\Gamma(\alpha) J_{t_2}^\alpha \eta'_{\varepsilon}(t) = \int_t^{t_2} (s-t)^{\alpha-1} \zeta_\varepsilon(s-t_1) \, ds - \int_t^{t_2} (s-t)^{\alpha-1} \zeta_\varepsilon(s-t_2) \, ds\\
&= \left\{
\begin{aligned}
&\int_0^{t_1-t} (t_1-t-s)^{\alpha-1} \zeta_\varepsilon(-s) \, ds - \int_0^{t_2-t} (t_2-t-s)^{\alpha-1} \zeta_\varepsilon(-s) \, ds, \,\, 0 \leq t \leq t_1,
\\
&- \int_0^{t_2-t} (t_2-t-s)^{\alpha-1} \zeta_\varepsilon(-s) \, ds, \,\, t_1 < t \leq t_2,
\end{aligned}
\right.
\end{align*}
where the last equality is due to the choice of $\zeta$ so that $\zeta_\varepsilon(s) = 0$ for $s \geq 0$.
That is, if we denote
$$
\tilde{\zeta}_\varepsilon(s) = \zeta_\varepsilon(-s),
$$
then
$$
J_{t_2}^\alpha \eta_\varepsilon'(t)
= \left\{
\begin{aligned}
(I_0^{\alpha} \tilde{\zeta}_\varepsilon)(t_1-t) - (I_0^{\alpha} \tilde{\zeta}_\varepsilon)(t_2-t), \quad 0 \leq t \leq t_1,
\\
- (I_0^{\alpha} \tilde{\zeta}_\varepsilon)(t_2-t), \quad t_1 < t \leq t_2,
\end{aligned}
\right.
$$
and
\begin{align*}
&\|J_{t_2}^\alpha \eta_\varepsilon'(t)\|_{L_{p'}(0,t_2)}^{p'} = \int_0^{t_1} \left|(I_0^{\alpha} \tilde{\zeta}_\varepsilon)(t_1-t) - (I_0^{\alpha} \tilde{\zeta}_\varepsilon)(t_2-t)\right|^{p'} \, dt\\
&\qquad + \int_{t_1}^{t_2} \left|(I_0^{\alpha} \tilde{\zeta}_\varepsilon)(t_2-t)\right|^{p'} \, dt := K_1 + K_2.
\end{align*}
By Lemma A.2 in \cite{MR3899965} along with $\alpha - 1 + 1/p' = \alpha - 1/p > 0$, we have
\begin{align}
							\label{eq0717_07}
K_2^{1/p'} &= \|I_0^\alpha \tilde{\zeta}_\varepsilon\|_{L_{p'}(0,t_2-t_1)} \leq  N (t_2-t_1)^{\alpha-1+1/p'}\|\tilde{\zeta}_\varepsilon\|_{L_1(0,t_2-t_1)}
\notag\\
&\leq N (t_2-t_1)^{\alpha-1/p},
\end{align}
where $N = N(\alpha,p)$.
To estimate $K_1$, we write
\begin{align*}
\Gamma(\alpha) & \left[(I_0^{\alpha} \tilde{\zeta}_\varepsilon)(t_1-t) - (I_0^{\alpha} \tilde{\zeta}_\varepsilon)(t_2-t)\right]
\\
= & \int_0^{t_1-t}(t_1-t-s)^{\alpha-1} \tilde{\zeta}_\varepsilon(s) \, ds - \int_0^{t_2-t}(t_2-t-s)^{\alpha-1}\tilde{\zeta}_\varepsilon(s) \, ds
\\
= & \int_0^{t_1-t}\left( (t_1-t-s)^{\alpha-1} - (t_2-t-s)^{\alpha-1}\right) \tilde{\zeta}_\varepsilon(s) \, ds
\\
& - \int_{t_1-t}^{t_2-t} (t_2-t-s)^{\alpha-1} \tilde{\zeta}_\varepsilon(s) \, ds.
\end{align*}
Thus,
\begin{align*}
K_1^{1/p'} &\leq N \left\|\int_0^{t_1-t}\left( (t_1-t-s)^{\alpha-1} - (t_2-t-s)^{\alpha-1}\right) \tilde{\zeta}_\varepsilon(s) \, ds\right\|_{L_{p'}(0,t_1)}\\
&\quad + N \left\|\int_{t_1-t}^{t_2-t} (t_2-t-s)^{\alpha-1} \tilde{\zeta}_\varepsilon(s) \, ds\right\|_{L_{p'}(0,t_1)}\\
&\leq N \int_0^{t_1}\left(\int_0^{t_1-s} \left| (t_1-t-s)^{\alpha-1}-(t_2-t-s)^{\alpha-1}\right|^{p'} \, dt \right)^{1/p'} \tilde{\zeta}_\varepsilon(s) \, ds\\
&\quad + N \int_0^{t_2} \left(\int_{t_1}^{t_2} \left|(t_2-t)^{\alpha-1}\right|^{p'} \, dt \right)^{1/p'} \, \tilde{\zeta}_\varepsilon(s) \, ds,
\end{align*}
where we used Minkowski's inequality for the second inequality.
Since $\alpha > 1-1/p'$, we readily see that
\begin{equation}
							\label{eq0717_01}
\left(\int_{t_1}^{t_2} (t_2-t)^{(\alpha-1)p'} \, dt\right)^{1/p'} = N(\alpha,p) (t_2-t_1)^{\alpha-1/p}.
\end{equation}
From the proof of Lemma A.14 in \cite{MR3899965} we also see that
$$
\left(\int_0^{t_1-s} \left| (t_1-t-s)^{\alpha-1}-(t_2-t-s)^{\alpha-1}\right|^{p'} \, dt \right)^{1/p'} \leq N(\alpha,p)(t_2-t_1)^{\alpha-1/p}.
$$
From this and \eqref{eq0717_01} together with the fact that the $L_1$-norm of $\tilde{\zeta}_\varepsilon$ on $\bR$ equals 1, we obtain that
$$
K_1^{1/p'} \leq N(\alpha,p)(t_2-t_1)^{\alpha-1/p},
$$
which together with \eqref{eq0717_07} proves \eqref{eq0717_06}.
The lemma is proved.
\end{proof}

We complete the proof of the H\"{o}lder continuity of functions in $\cH_{p,0}^{\alpha,1}$ for $p > d + 2/\alpha$.

\begin{theorem}[H\"{o}lder continuity]
							\label{thm0717_1}
Let $T \in (0,\infty)$, $\alpha \in (0,1)$, and $p \in (1,\infty)$ such that
$$
\sigma := 1 - (d+2/\alpha)/p > 0.
$$
Then, for $u \in \cH_{p,0}^{\alpha,1}(\bR^d_T)$, we have
$$
[u]_{C^{\sigma \alpha/2,\sigma}(\bR^d_T)} \leq N(d,\alpha,p) \|u\|_{\cH_p^{\alpha,1}(\bR^d_T)}.
$$
\end{theorem}

\begin{proof}
By mollifications, we may assume that $u \in C_0^\infty\left([0,T] \times \bR^d\right)$ such that $u(0,x) = 0$ and
$$
\int_0^T \int_{\bR^d} I_0^{1-\alpha} u \varphi_t \, dx \, dt = \int_0^T \int_{\bR^d} g_i D_i \varphi \, dx \, dt - \int_0^T \int_{\bR^d} f \varphi \, dx \, dt
$$
for any $\varphi \in C_0^\infty\left([0,T) \times \bR^d\right)$, where $g_i, f \in L_p(\bR^d_T)$.
Define
$$
K = \sup\left\{\frac{|u(t,x) - u(s,y)|}{|t-s|^{\frac{\sigma \alpha}2} + |x-y|^\sigma}:
(t,x),(s,y) \in \bR^d_T,0<|t_1-t_2|^{\frac \alpha 2} + |x-y|\le 1\right\}.
$$
To prove the estimate, we take $(t_1,x), (t_2,y) \in \bR^d_T$ and set
$$
\rho = \varepsilon \left( |t_1-t_2|^{\alpha/2} + |x-y| \right) < 1,
$$
where $\varepsilon \in (0,1)$ is to be specified below.
We write
$$
u(t_1,x) - u(t_2,y) = \left(u(t_1,x) - u(t_2,x)\right) + \left(u(t_2,x) - u(t_2,y)\right) := J_1 + J_2.
$$
To estimate $J_1$, for $z \in B_\rho(x)$, we have
\begin{equation}
							\label{eq0801_01}
\begin{aligned}
J_1 &= \left(u(t_1,x)-u(t_1,z)\right) + \left(u(t_1,z) - u(t_2,z)\right) + \left(u(t_2,z) - u(t_2,x)\right)
\\
&\leq 2 K \rho^\sigma + \left(u(t_1,z) - u(t_2,z)\right).
\end{aligned}
\end{equation}
Take $\phi(z) \in C_0^\infty\left(B_\rho(x)\right)$ such that $\phi(z) \geq 0$,
\begin{equation}
							\label{eq0717_08}
\int_{B_\rho(x)} \phi(z) \, dz = 1, \quad \|\phi\|_{L_{p'}\left(B_\rho(x)\right)} = N \rho^{-d/p}, \quad \|D\phi\|_{L_{p'}\left(B_\rho(x)\right)} = N\rho^{-d/p-1},
\end{equation}
where $N = N(d,p)$.
Then by multiplying both sides of the  inequality \eqref{eq0801_01} by $\phi$ and integrating over $B_\rho(x)$, we have
$$
J_1 \leq 2 K \rho^\sigma + \int_{B_\rho(x)} \left(u(t_1,z) - u(t_2,z)\right) \phi(z) \, dz.
$$
Using this and a corresponding inequality for $-J_1$, it follows from Lemma \ref{lem0717_01}, \eqref{eq0717_08}, and the choice of $\rho$ that
\begin{align*}
|J_1| &\leq 2 K \rho^\sigma + N |t_2-t_1|^{\alpha-1/p} \rho^{-d/p-1} \left(\|g\|_{L_p(\bR^d_T)} + \|f\|_{L_p(\bR^d_T)}\right)\\
&\leq 2 K \rho^\sigma + N \varepsilon^{-2+2/(\alpha p)} \rho^\sigma \left(\|g\|_{L_p(\bR^d_T)} + \|f\|_{L_p(\bR^d_T)}\right),
\end{align*}
where $N = N(d,\alpha,p)$.

To estimate $|J_2|$, we have that, for $s \in (t_2-\rho^{1/\alpha},t_2+\rho^{1/\alpha}) \cap (0,T)$,
\begin{align*}
|J_2| &\leq |u(t_2,x) - u(s,x)| + |u(s,x) - u(s,y)| + |u(s,y) - u(t_2, y)|\\
&\leq
2 K \rho^\sigma + N(d,p)|x-y|^{1-d/p} \|u(s,\cdot)\|_{W_p^1(\bR^d)},
\end{align*}
where we used the usual Sobolev embedding for functions in $x \in \bR^d$ and the condition that $1-d/p > 0$.
Then by taking the average of $J_2$ over the interval $(t_2 - \rho^{2/\alpha}, t_2+\rho^{2/\alpha}) \cap (0,T)$ with respect to $s$ along with H\"{o}lder's inequality, we get
$$
|J_2| \leq 2K \rho^\sigma + N \varepsilon^{-1+d/p} \rho^\sigma \||u|+|D_x u|\|_{L_p(\bR^d_T)}.
$$

Collecting the estimates for $J_1$ and $J_2$ above, we see that
$$
|u(t_1,x) - u(t_2,y)| \leq 4 K \rho^\sigma + N \left(\varepsilon^{-2+2/(\alpha p)} +  \varepsilon^{-1+d/p}  \right) \rho^\sigma \|u\|_{\cH_p^{\alpha,1} (\bR^d_T)},
$$
which implies that
$$
K \leq 4 \varepsilon^\sigma K + N(d,\alpha,p) (\varepsilon^{-1-d/p} +  \varepsilon^{-2/(\alpha p)} )\|u\|_{\cH_p^{\alpha,1} (\bR^d_T)}.
$$
By choosing $\varepsilon > 0$ small enough so that
$4 \varepsilon^\sigma < 1$, we obtain
\begin{equation}
                            \label{eq8.11}
K\le N(d,\alpha,p)\|u\|_{\cH_p^{\alpha,1} (\bR^d_T)}.
\end{equation}
To finish the proof, it suffices to show that
\begin{equation}
                                    \label{eq8.12}
\|u\|_{L_\infty(\bR^d_T)} \leq N(d,\alpha,p) \|u\|_{\cH_p^{\alpha,1}(\bR^d_T)}.
\end{equation}
Indeed, for any $(t,x),(s,y)\in \bR^d_T$, by the triangle inequality we have
$$
|u(t,x)|\le |u(t,x)-u(s,y)|+|u(s,y)|.
$$
Taking the averages of both sides of the above inequality with respect to $(s,y)$ over the set
$$
\{(s,y)\in \bR^d_T:|t-s|^{\alpha/2} + |x-y|\le 1\}
$$
and using \eqref{eq8.11}, we get \eqref{eq8.12}. The lemma is proved.
\end{proof}

\begin{corollary}
                                        \label{cor1029_1}
Let $T \in (0,\infty)$, $0 < r < R < \infty$, $\alpha \in (0,1)$, and $p \in (1,\infty)$ such that
$$
\sigma := 1 - (d+2/\alpha)/p > 0.
$$
Then, for $u \in \cH_{p,0}^{\alpha,1}\left((0,T) \times B_R\right)$, we have
\begin{equation}
								\label{eq0224_01}
[u]_{C^{\sigma \alpha/2, \sigma}\left((0,T) \times B_r\right)} \leq N(d,\alpha,p,r,R) \|u\|_{\cH_p^{\alpha,1}\left((0,T) \times B_R\right)}.
\end{equation}
\end{corollary}

\begin{proof}
Take an infinitely differentiable function $\psi(x)$ defined on $\bR^d$ such that $\psi(x) = 1$ on $B_r$ and $\psi(x) = 0$ on $\bR^d \setminus B_R$, and consider $\psi u$, which belongs to $\cH_{p,0}^{\alpha,1}(\bR^d_T)$.
In particular,
$$
\partial_t^\alpha (\psi u) = D_i (\psi g_i) + \psi f - g_i D_i \psi
$$
in $\bR^d$ if
$$
\partial_t^\alpha u = D_i g_i + f
$$
in $(0,T) \times B_R$.
Then by Theorem \ref{thm0717_1}, we have
$$
\|u\|_{C^{\sigma\alpha/2,\sigma}\left((0,T) \times B_r\right)} \leq N \|\psi u\|_{C^{\sigma\alpha/2,\sigma}(\bR^d_T)} \leq N \|\psi u\|_{\cH_p^{\alpha,1}(\bR^d_T)},
$$
where $N=N(d,\alpha,p)$.
We then obtain the desired estimate upon noting that the last term in the above inequality is bounded by the right-hand side of \eqref{eq0224_01}, where the constant depends on $r$ and $R$ as well.
\end{proof}

We now prove Theorem \ref{thm0811_01}, which is about Sobolev embeddings.

\begin{theorem}[Sobolev embedding]
							\label{thm0811_01}
Let $T \in (0,\infty)$, $\alpha \in (0,1)$, and $p, p_0\in (1,\infty)$ such that $p_0 > p$ and
\begin{equation}
							\label{eq0728_01}
1 - \frac{d+2/\alpha}{p} \geq - \frac{d+2/\alpha}{p_0}.
\end{equation}
Then, for $u \in \cH_{p,0}^{\alpha,1}(\bR^d_T)$,
we have
\begin{equation}
							\label{eq0804_02}
\|u\|_{L_{p_0}(\bR^d_T)} \leq N \|u\|_{\cH_p^{\alpha,1}(\bR^d_T)},
\end{equation}
where $N = N(d,\alpha,p,p_0)$.
More precisely,
\begin{equation}
							\label{eq0730_01}
\begin{aligned}
\|u\|_{L_{p_0}(\bR^d_T)} &\leq N \varepsilon^{\frac{d + 2/\alpha}{p_0} - \frac{d + 2/\alpha}{p}} \|u\|_{L_p(\bR^d_T)} + N \varepsilon^{2+ \frac{d+2/\alpha}{p_0} - \frac{d+2/\alpha}{p}} \|f\|_{L_p(\bR^d_T)}
\\
&+ N \varepsilon^{1+ \frac{d+2/\alpha}{p_0} - \frac{d+2/\alpha}{p}} \left(\|Du\|_{L_p(\bR^d_T)} + \|g\|_{L_p(\bR^d_T)} \right)
\end{aligned}
\end{equation}
for any $\varepsilon > 0$, where $N = N(d,\alpha,p,p_0)$
provided that
\begin{equation}
							\label{eq0730_02}
\int_0^T \int_{\bR^d} I_0^{1-\alpha} u \varphi_t \, dx \, d s = \int_0^T \int_{\bR^d} g_i D_i \varphi \, dx \, d s - \int_0^T \int_{\bR^d} f \varphi \, dx \, d s
\end{equation}
for any $\varphi \in C_0^\infty\left([0,T) \times \bR^d\right)$, where $g = (g_1,\ldots,g_d), f \in L_p(\bR^d_T)$.
In particular, if the equality holds in \eqref{eq0728_01}, we have
\begin{equation}
							\label{eq0804_01}
\|u\|_{L_{p_0}(\bR^d_T)} \leq N \|u\|^{1/2}_{L_p(\bR^d_T)}\|f\|_{L_p(\bR^d_T)}^{1/2} + N \left(\|Du\|_{L_p(\bR^d_T)} + \|g\|_{L_p(\bR^d_T)} \right).
\end{equation}
\end{theorem}

\begin{proof}
It suffices to prove \eqref{eq0730_01} for $u \in C^\infty\left([0,T] \times \bR^d\right)$ such that $u(0,x) = 0$ and $u(t,x)$ vanishes for large $|x|$.
Throughout the proof, $u(t,x)$ is extended to be zero for $t \leq 0$.
Note that the equality \eqref{eq0730_02} holds with $t \in [0,T]$ in place of $T$ if $\varphi \in C_0^\infty\left([0,t) \times \bR^d\right)$.

Set $\eta(t)$ to be an infinitely differentiable function defined on $\bR$ such that $\eta(t) \geq 0$, $\eta(t) = 0$ for $t \notin (1/2,1)$, and
$$
\int_\bR \eta(t) \, dt = 1.
$$
Also set $\psi(x)$ to be an infinitely differentiable function defined on $\bR^d$ such that $\psi(x) \geq 0$, $\psi(x) = 0$ for $x \in \bR^d \setminus B_1$, and
$$
\int_{\bR^d} \psi(x) \, dx = 1.
$$
Then, for $(t,x) \in \bR^d_T$, we define
\begin{align*}
u^{(\varepsilon)}(t,x) &= \int_0^T\int_{\bR^d} u(s,y) \eta_{\varepsilon^{2/\alpha}}(t-s) \psi_\varepsilon(x-y) \, dy \, ds
\\
&= \int_{-\infty}^t \int_{\bR^d} u(s,y) \eta_{\varepsilon^{2/\alpha}}(t-s) \psi_\varepsilon(x-y) \, dy \, ds,
\end{align*}
where
$$
\eta_{\varepsilon^{2/\alpha}}(t) = \varepsilon^{-2/\alpha} \eta(t/\varepsilon^{2/\alpha}), \quad \psi_\varepsilon(x) = \varepsilon^{-d} \psi(x/\varepsilon).
$$
Set $a \in [1,\infty]$ so that
\begin{equation}
							\label{eq0802_01}
\frac{1}{p} + \frac{1}{a} = \frac{1}{p_0} + 1.
\end{equation}
Indeed, we see that $a \in [1,p_0]$ because $1 \leq p \leq p_0$.
By Young's convolution inequality it follows that
\begin{equation}
							\label{eq0803_04}
\|u^{(\varepsilon)}\|_{L_{p_0}(\bR^d_T)} \leq N \varepsilon^{\frac{d + 2/\alpha}{p_0} - \frac{d + 2/\alpha}{p}} \|u\|_{L_p(\bR^d_T)}.
\end{equation}

For $(t,x) \in \bR^d_T$, we now consider
$$
u(t,x) - u^{(\varepsilon)}(t,x) = \int_{-\infty}^t\int_{\bR^d} \left(u(t,x) - u(s,y)\right) \eta_{\varepsilon^{2/\alpha}}(t-s) \psi_\varepsilon(x-y) \, dy \, ds.
$$
Let
$$
\tau(\lambda) = \left( (1-\lambda^{2/\alpha}) t + \lambda^{2/\alpha} s, (1-\lambda) x + \lambda y \right), \quad \lambda \in [0,1].
$$
We then write
\begin{align*}
&u(t,x)-u(s,y) = u\left(\tau(0)\right) - u\left(\tau(1)\right)\\
&= \int_0^1 \frac{2}{\alpha} \lambda^{\frac{2}{\alpha}-1} (t-s) u_t \left( \tau(\lambda) \right) \, d\lambda + \int_0^1 \nabla u\left(\tau(\lambda)\right) \cdot (x-y) \, d\lambda.
\end{align*}
Hence,
\begin{equation}
							\label{eq0803_03}
\begin{aligned}
&u(t,x) - u^{(\varepsilon)}(t,x)
\\
&= \frac{2}{\alpha} \int_{-\infty}^t \int_{\bR^d} \int_0^1 \lambda^{\frac{2}{\alpha}-1} (t-s) u_t \left( \tau(\lambda) \right) \eta_{\varepsilon^{2/\alpha}}(t-s) \psi_\varepsilon(x-y) \, d\lambda  \, dy \, ds
\\
&\quad \, + \int_{-\infty}^t \int_{\bR^d} \int_0^1 (x-y) \cdot \nabla u\left(\tau(\lambda)\right)  \eta_{\varepsilon^{2/\alpha}}(t-s) \psi_\varepsilon(x-y) \, d\lambda  \, dy \, ds
\\
&=: K_1 + K_2.
\end{aligned}
\end{equation}

To estimate $K_1$, we denote
$$
\zeta(t) := t \eta(t).
$$
By the change of variables $\left( (1-\lambda^{2/\alpha}) t + \lambda^{2/\alpha} s, (1-\lambda) x + \lambda y \right) \to (s,y)$,
\begin{align*}
K_1 &= \frac{2}{\alpha} \varepsilon^{-d} \int_0^1 \lambda^{-d-1} \int_{-\infty}^t \int_{\bR^d} u_t(s,y) \, \zeta\left( \frac{t-s}{\varepsilon^{2/\alpha}\lambda^{2/\alpha}}\right) \psi \left( \frac{x-y}{\varepsilon \lambda}\right) \, dy \, ds \, d\lambda\\
&= \frac{2}{\alpha} \int_0^\varepsilon \lambda^{-d-1} \int_{-\infty}^t \int_{\bR^d} u_t(s,y) \, \zeta\left( \frac{t-s}{\lambda^{2/\alpha}}\right) \psi \left( \frac{x-y}{\lambda}\right) \, dy \, ds \, d\lambda,
\end{align*}
where the second equality is due to the change of variable $\varepsilon \lambda \to \lambda$.
For each $t \in [0,T]$, by Lemma \ref{lem0716_1} with $t$ in place of $T$ and the fact that, as a function of $s$, $\zeta\left(\frac{t-s}{\lambda^{2/\alpha}}\right) = 0$ at $s=t$ and $u(s,y) = 0$ for $s \leq 0$,
\begin{equation}
							\label{eq0719_01}
\begin{aligned}
\int_{-\infty}^t u_t(s,y) \,  \zeta\left( \frac{t-s}{\lambda^{2/\alpha}}\right)\, ds &= \int_0^t u_t(s,y) \, \zeta\left(\frac{t-s}{\lambda^{2/\alpha}}\right) \, ds
\\
&= \int_0^t I_0^{1-\alpha}u(s,y) \, \partial_s \left( J_t^\alpha \tilde{\zeta}'(s) \right) \, ds,
\end{aligned}
\end{equation}
where
$$
\tilde{\zeta}(s) = \zeta\left(\frac{t-s}{\lambda^{2/\alpha}}\right), \quad J_t^\alpha \tilde{\zeta}'(s) \in C^1\left([0,t]\right).
$$
By setting
$$
H(t) := I_0^\alpha \zeta' (t) = \frac{1}{\Gamma(\alpha)} \int_0^t (t-r)^{\alpha-1} \zeta'(r) \, dr,
$$
we see that
\begin{equation}
							\label{eq0726_02}
\begin{aligned}
J_t^\alpha \tilde{\zeta}'(s) &= - \lambda^{-2/\alpha} \frac{1}{\Gamma(\alpha)} \int_s^t (r-s)^{\alpha-1} \zeta'\left(\frac{t-r}{\lambda^{2/\alpha}}\right) \, dr
\\
&= - \lambda^{-2/\alpha} \frac{1}{\Gamma(\alpha)} \int_0^{t-s} (t-s-r)^{\alpha-1} \zeta'\left(\frac{r}{\lambda^{2/\alpha}}\right) \, dr
\\
& = - \lambda^{2-2/\alpha} H\left(\frac{t-s}{\lambda^{2/\alpha}}\right),
\end{aligned}
\end{equation}
where $H$ satisfies
\begin{equation}
							\label{eq0803_06}
H(t) \leq N(\alpha) \quad \text{for} \,\, t \in (0,2), \quad
H(t) \leq N(\alpha) t^{\alpha-2} \quad \text{for} \,\, t \geq 2.
\end{equation}
Indeed, the first inequality in \eqref{eq0803_06} follows from Lemma A.2 in \cite{MR3899965} so that
$$
\|H\|_{L_\infty(0,2)} = \|I_0^\alpha \zeta'\|_{L_\infty(0,2)} \leq N \|\zeta'\|_{L_\mu(0,2)},
$$
where $\mu \in (1,\infty)$ with $\alpha > 1/\mu$.
The second inequality in \eqref{eq0803_06} follows from
$$
\Gamma(\alpha) H(t) = \int_0^1 (t-r)^{\alpha-1} \zeta'(r) \, dr = (\alpha-1) \int_0^1 (t-r)^{\alpha-2} \zeta (r) \, dr \leq N t^{\alpha-2}
$$
provided that $t \geq 2$, where we used integration by parts along with the fact that $\zeta(r) = 0$ for $r \geq 1$ or $r \leq 0$.
Considering $(\alpha-2)a+1 < 0$, the inequalities in \eqref{eq0803_06}
imply that
\begin{equation}
							\label{eq0803_07}
\|H\|_{L_a(0,\infty)} \leq N(\alpha,p,p_0).
\end{equation}

From \eqref{eq0719_01} and \eqref{eq0730_02} with $t$ in place of $T$, it follows that
\begin{align*}
&\int_{-\infty}^t \int_{\bR^d} u_t(s,y) \, \zeta\left( \frac{t-s}{\lambda^{2/\alpha}}\right) \psi \left( \frac{x-y}{\lambda}\right) \, dy \, ds
\\
&= \int_0^t \int_{\bR^d} I_0^{1-\alpha}u(s,y) \, \partial_s \left[ J_t^\alpha \tilde{\zeta}'(s)\psi \left( \frac{x-y}{\lambda}\right)\right] \, dy \, ds
\\
&= - \lambda^{-1} \int_0^t \int_{\bR^d} g_i(s,y) J_t^\alpha \tilde{\zeta}'(s) \, (D_i\psi) \left( \frac{x-y}{\lambda}\right) \, dy \, ds
\\
& \quad \, - \int_0^t \int_{\bR^d} f(s,y) J_t^\alpha \tilde{\zeta}'(s) \, \psi \left( \frac{x-y}{\lambda}\right) \, dy \, ds.
\end{align*}
Thus,
\begin{align*}
K_1 & = \frac{2}{\alpha} \int_0^\varepsilon \lambda^{-d-\frac{2}{\alpha}} \int_0^t \int_{\bR^d} g_i(s,y) H\left(\frac{t-s}{\lambda^{2/\alpha}}\right) (D_i \psi)\left(\frac{x-y}{\lambda}\right) \, dy \, ds \, d\lambda
\\
& \quad \, + \frac{2}{\alpha} \int_0^\varepsilon \lambda^{-d+1-\frac{2}{\alpha}} \int_0^t \int_{\bR^d} f(s,y) H\left(\frac{t-s}{\lambda^{2/\alpha}}\right) \psi \left( \frac{x-y}{\lambda}\right) \, dy \, ds \, d\lambda
\\
& := K_{1,1}(t,x) + K_{1,2}(t,x).
\end{align*}
To estimate $K_{1,1}$, upon recalling \eqref{eq0802_01} with Young's convolution inequality with respect to $x \in \bR^d$ and Minkowski's inequality, we have
\begin{equation}
							\label{eq0803_01}
\|K_{1,1}(t,\cdot)\|_{L_{p_0}(\bR^d)}
\le N \int_0^\varepsilon \lambda^{-d-\frac{2}{\alpha}+\frac{d}{a}} \int_0^t \left|H\left(\frac{t-s}{\lambda^{\frac{2}{\alpha}}}\right)\right| \|g(s,\cdot)\|_{L_p(\bR^d)} \, ds \, d\lambda,
\end{equation}
where we used \eqref{eq0726_02} and $N = N(d,\alpha,p,p_0)$.
We now consider two cases.

{\bf Case 1 for $K_{1,1}$}: $1-(d+2/\alpha)/p > - (d+2/\alpha)/p_0$.
In this case by taking the $L_{p_0}$-norms of both sides of the inequality in \eqref{eq0803_01} along with Young's convolution inequality with respect to $t \in (0,T)$, we have
\begin{align*}
&\|K_{1,1}\|_{L_{p_0}(\bR^d_T)} = \left\| \|K_{1,1}(t,\cdot)\|_{L_{p_0}(\bR^d)} \right\|_{L_{p_0}(0,T)}\\
&\leq N \|g\|_{L_{p}(\bR^d_T)} \int_0^\varepsilon \lambda^{-d-2/\alpha+d/a} \left\|H\left(\frac{\cdot}{\lambda^{2/\alpha}}\right)\right\|_{L_a(0,\infty)} \, d\lambda,
\end{align*}
where by \eqref{eq0803_07}
$$
\left\|H\left(\frac{\cdot}{\lambda^{2/\alpha}}\right)\right\|_{L_a(0,\infty)}
\leq N(\alpha,p,p_0)\lambda^{2/(\alpha a)}.
$$
Thus, because
$$
-d-2/\alpha+d/a+2/(\alpha a) = d/p_0 - d/p + (1/p_0-1/p)2/\alpha > - 1,
$$
we arrive at
\begin{equation}
							\label{eq0804_03}
\begin{aligned}
\|K_{1,1}\|_{L_{p_0}(\bR^d_T)} &\leq N \|g\|_{L_{p}(\bR^d_T)} \int_0^\varepsilon \lambda^{-d-\frac{2}{\alpha}+\frac{d}{a}+\frac{2}{\alpha a}} \, d\lambda
\\
&= N \varepsilon^{1 + \frac{d}{p_0} - \frac{d}{p} + \frac{2}{\alpha}\left(\frac{1}{p_0}-\frac{1}{p}\right)}\|g\|_{L_{p}(\bR^d_T)},
\end{aligned}
\end{equation}
where $N = N(d,\alpha,p,p_0)$.

{\bf Case 2 for $K_{1,1}$}: $1-(d+2/\alpha)/p = - (d+2/\alpha)/p_0$.
In this case, we change the order of the integrations in \eqref{eq0803_01} and use the following estimate from \eqref{eq0803_06}
$$
\left| H\left(\frac{t-s}{\lambda^{\frac{2}{\alpha}}}\right)\right| \leq
\left\{
\begin{aligned}
N \left(\frac{t-s}{\lambda^{\frac{2}{\alpha}}}\right)^{\alpha-2} \quad &\text{for} \,\, 0 < \lambda < \left(\frac{t-s}{2}\right)^{\frac{\alpha}{2}},
\\
N \quad \quad \quad \quad &\text{for} \,\, \lambda \geq \left(\frac{t-s}{2}\right)^{\frac{\alpha}{2}}.
\end{aligned}
\right.
$$
Thus, the integral with respect to $\lambda$ in \eqref{eq0803_01} is estimated as
\begin{align*}
&\int_0^\varepsilon \lambda^{-d-\frac{2}{\alpha}+\frac{d}{a}} \left| H\left(\frac{t-s}{\lambda^{\frac{2}{\alpha}}}\right)\right| \, d\lambda \leq N \int_0^{\left(\frac{t-s}{2}\right)^{\frac{\alpha}{2}}} \lambda^{-d-\frac{2}{\alpha}+\frac{d}{a}} \left(\frac{t-s}{\lambda^{\frac{2}{\alpha}}}\right)^{\alpha-2} \, d\lambda\\
&\qquad + N \int_{\left(\frac{t-s}{2}\right)^{\frac{\alpha}{2}}}^\infty \lambda^{-d-\frac{2}{\alpha}+\frac{d}{a}} \, d\lambda = N (t-s)^{-1+ \frac{1}{p}- \frac{1}{p_0}},
\end{align*}
where the calculation relies on the equality $1-(d+2/\alpha)/p = - (d+2/\alpha)/p_0$, $\alpha \in (0,1)$, and $1/p > 1/p_0$, so that
\begin{align*}
&-d-2/\alpha+d/a-(\alpha-2)2/\alpha = 2/\alpha + d/p_0 - d/p - 2\\
&= 2/\alpha + \frac{2}{\alpha}(1/p - 1/p_0) - 3 > -1
\end{align*}
and
$$
-d-2/\alpha+d/a = -2/\alpha+d/p_0 - d/p < -1.
$$
Hence,
$$
\|K_{1,1}(t,\cdot)\|_{L_{p_0}(\bR^d)} \leq N \int_0^t (t-s)^{-1+1/p - 1/p_0} \|g(s,\cdot)\|_{L_p(\bR^d)} \, ds.
$$
Since $1/p - 1/p_0 \in (0,1)$, the Hardy-Littlewood-Sobolev theorem of fractional integration (see, for instance, \cite[p.119, Theorem 1]{MR0290095}) applied to the above inequality gives that
\begin{equation}
							\label{eq0804_04}
\begin{aligned}
&\|K_{1,1}\|_{L_{p_0}(\bR^d_T)}
= \left\| \|K_{1,1}(t,\cdot)\|_{L_{p_0}(\bR^d)}\right\|_{L_{p_0}(0,T)}
\\
&\leq N \left\| \int_{-\infty}^\infty |t-s|^{-1+\frac{1}{p}-\frac{1}{p_0}} \|g(s,\cdot)\|_{L_p(\bR^d)} 1_{0<s<T} \, ds \right\|_{L_{p_0}(\bR)} \leq N \|g\|_{L_p(\bR^d_T)},
\end{aligned}
\end{equation}
where $N = N(d, \alpha, p, p_0)$.

To estimate $K_{1,2}$, for both cases
$$
1-(d+2/\alpha)/p > - (d+2/\alpha)/p_0 \quad \text{and} \quad 1-(d+2/\alpha)/p = - (d+2/\alpha)/p_0,
$$ similarly as in Case 1 for $K_{1,1}$, we have
\begin{equation}
							\label{eq0804_05}
\begin{aligned}
&\|K_{1,2}\|_{L_{p_0}(\bR^d_T)}
= \left\|\|K_{1,2}(t,\cdot)\|_{L_{p_0}(\bR^d)} \right\|_{L_{p_0}(0,T)}
\\
&\leq N \int_0^\varepsilon \lambda^{-d+1-\frac{2}{\alpha}+\frac{d}{a}} \left\|H\left(\frac{\cdot}{\lambda^{\frac{2}{\alpha}}}\right)\right\|_{L_a(0,\infty)} \|f\|_{L_p(\bR^d_T)} \, d\lambda
\\
&\leq N \|f\|_{L_p(\bR^d_T)} \int_0^\varepsilon \lambda^{1+ \frac{d}{p_0} - \frac{d}{p}+\frac{2}{\alpha}\left(\frac{1}{p_0}-\frac{1}{p}\right)}\, d\lambda
\\
&= N \varepsilon^{2+\frac{d}{p_0} - \frac{d}{p}+\frac{2}{\alpha}\left(\frac{1}{p_0}-\frac{1}{p}\right)} \|f\|_{L_p(\bR^d_T)},
\end{aligned}
\end{equation}
where $1+d/p_0 - d/p+(1/p_0-1/p)2/\alpha \geq 0$ and  $N = N(d,\alpha,p,p_0)$.

The estimate of $K_2$ is almost the same as $K_{1,1}$ with $\nabla u$ and $\eta$ in place of $g$ and $|H|$, respectively.
If $1-(d+2/\alpha)/p > -(d+2/\alpha)/p_0$, by following the calculation in Case 1 for $K_{1,1}$,
we have
\begin{equation}
							\label{eq0806_01}
\|K_2\|_{L_{p_0}(\bR^d_T)} \leq N \varepsilon^{1 + d/p_0 - d/p + (1/p_0-1/p)2/\alpha}\|Du\|_{L_{p}(\bR^d_T)},
\end{equation}
where $N = N(d,\alpha,p,p_0)$.
If $1-(d+2/\alpha)/p = -(d+2/\alpha)/p_0$, we follow the calculation in Case 2 for $K_{1,1}$ to get
\begin{equation}
							\label{eq0806_02}
\|K_2\|_{L_{p_0}(\bR^d_T)} \leq N \|Du\|_{L_p(\bR^d_T)},
\end{equation}
where $N = N(d,\alpha,p,p_0)$.

From the estimates \eqref{eq0806_01}, \eqref{eq0806_02} for $K_2$, and the estimates \eqref{eq0804_03}, \eqref{eq0804_04}, \eqref{eq0804_05} for $K_{1,1}$ and $K_{1,2}$ along with \eqref{eq0803_04} and \eqref{eq0803_03}, we obtain the inequality \eqref{eq0730_01}, which implies \eqref{eq0804_02} by the choice of, for instance, $\varepsilon = 1$.
In particular, if $1 + (d+2/\alpha)/p_0 = (d+2/\alpha)/p$, the inequality \eqref{eq0730_01} becomes
$$
\|u\|_{L_{p_0}(\bR^d_T)} \leq N \varepsilon^{-1}\|u\|_{L_p(\bR^d_T)} + N \varepsilon \|f\|_{L_p(\bR^d_T)} + N \left(\|Du\|_{L_p(\bR^d_T)} + \|g\|_{L_p(\bR^d_T)} \right).
$$
Minimizing the right-hand side with respect to $\varepsilon > 0$, we arrive at \eqref{eq0804_01}.
The theorem is proved.
\end{proof}

\begin{corollary}
							\label{cor0811_1}
Let $\alpha \in (0,1)$, $T \in (0,\infty)$, $0 < r < R < \infty$, and $p,p_0 \in (1,\infty)$ satisfy $p_0 > p$ and \eqref{eq0728_01}.
For $u \in \cH_{p,0}^{\alpha,1}\left((0,T) \times B_R\right)$, we have
\begin{equation*}
				%			\label{eq0811_01}
\|u\|_{L_{p_0}\left((0,T) \times B_r \right)} \leq N \|u\|_{\cH_p^{\alpha,1}\left((0,T) \times B_R\right)},
\end{equation*}
where $N = N(d,\alpha,p,p_0,r,R)$.
\end{corollary}

\begin{proof}
The proof is the same as that of Corollary \ref{cor1029_1} with $C^{\sigma\alpha/2,\sigma}$ replaced by $L_{p_0}$.
\end{proof}

\bibliographystyle{plain}

\def\cprime{$'$}

\end{document}